\newfont{\bcb}{msbm10}
\newfont{\matb}{cmbx10}
\newfont{\got}{eufm10}

\documentclass[12pt]{amsart}
\usepackage{amsmath, amsthm, amscd, amsfonts, amssymb, latexsym, graphicx, color}
\usepackage[bookmarksnumbered, colorlinks, plainpages, hypertex]{hyperref}

\usepackage[cp1250]{inputenc}

\usepackage{amsmath,amsthm}
\usepackage{amssymb,latexsym}
\usepackage{enumerate}

\newtheorem{theorem}{Theorem}[section]
\newtheorem{lemma}[theorem]{Lemma}
\newtheorem{proposition}[theorem]{Proposition}
\newtheorem{corollary}[theorem]{Corollary}
\theoremstyle{definition}

\theoremstyle{remark}
\newtheorem{remark}[theorem]{Remark}
\numberwithin{equation}{section}

\begin{document}

\title[A closedness theorem and its applications]{A closedness theorem and applications\\
       in geometry of rational points\\
       over Henselian valued fields}

\author[Krzysztof Jan Nowak]{Krzysztof Jan Nowak}





\dedicatory{Dedicated to Goo Ishikawa on the occasion of his 60th
birthday}

\begin{abstract}
We develop geometry of algebraic subvarieties of $K^{n}$ over
arbitrary Henselian valued fields $K$ of equicharacteristic zero.
This is a continuation of our previous article concerned with
algebraic geometry over rank one valued fields. At the center of
our approach is again the closedness theorem to the effect that
the projections $K^{n} \times \mathbb{P}^{m}(K) \to K^{n}$ are
definably closed maps. It enables, in particular, application of
resolution of singularities in much the same way as over locally
compact ground fields. As before, the proof of that theorem uses,
among others, the local behavior of definable functions of one
variable and fiber shrinking, being a relaxed version of curve
selection. But now, to achieve the former result, we first examine
functions given by algebraic power series. All our previous
results will be established here in the general settings: several
versions of curve selection (via resolution of singularities) and
of the \L{}ojasiewicz inequality (via two instances of quantifier
elimination indicated below), extending continuous hereditarily
rational functions as well as the theory of regulous functions,
sets and sheaves, including Nullstellensatz and Cartan's theorems
A and B. Two basic tools are quantifier elimination for Henselian
valued fields due to Pas and relative quantifier elimination for
ordered abelian groups (in a many-sorted language with imaginary
auxiliary sorts) due to Cluckers--Halupczok. Other, new
applications of the closedness theorem are piecewise continuity of
definable functions, H\"{o}lder continuity of functions definable
on closed bounded subsets of $K^{n}$, the existence of definable
retractions onto closed definable subsets of $K^{n}$ and a
definable, non-Archimedean version of the Tietze--Urysohn
extension theorem. In a recent paper, we established a version of
the closedness theorem over Henselian valued fields with analytic
structure along with several applications.



\end{abstract}

\maketitle

\begin{center}
\end{center}

\section{Introduction}

Throughout the paper, $K$ will be an arbitrary Henselian valued
field of equicharacteristic zero with valuation $v$, value group
$\Gamma$, valuation ring $R$ and residue field $\Bbbk$. Examples
of such fields are the quotient fields of the rings of formal
power series and of Puiseux series with coefficients from a field
$\Bbbk$ of characteristic zero as well as the fields of Hahn
series (maximally complete valued fields also called
Malcev--Neumann fields; cf.~\cite{Kap}):
$$ \Bbbk((t^{\Gamma})) := \left\{ f(t) = \sum_{\gamma \in \Gamma} \
   a_{\gamma}t^{\gamma} : \ a_{\gamma} \in \Bbbk, \ \text{supp}\,
   f(t) \ \text{is well ordered} \right\}. $$
We consider the ground field $K$ along with the three-sorted
language $\mathcal{L}$ of Denef--Pas (cf.~\cite{Pa1,Now2}). The
three sorts of $\mathcal{L}$ are: the valued field $K$-sort, the
value group $\Gamma$-sort and the residue field $\Bbbk$-sort. The
language of the $K$-sort is the language of rings; that of the
$\Gamma$-sort is any augmentation of the language of ordered
abelian groups (and $\infty$); finally, that of the $\Bbbk$-sort
is any augmentation of the language of rings. The only symbols of
$\mathcal{L}$ connecting the sorts are two functions from the main
$K$-sort to the auxiliary $\Gamma$-sort and $\Bbbk$-sort: the
valuation map and an angular component map.

\vspace{1ex}

Every valued field $K$ has a topology induced by its valuation
$v$. Cartesian products $K^{n}$ are equipped with the product
topology, and their subsets inherit a topology, called the
$K$-topology. This paper is a continuation of our
paper~\cite{Now2} devoted to geometry over Henselian rank one
valued fields, and includes our recent
preprints~\cite{Now3,Now4,Now5}. The main aim is to prove (in
Section~8) the closedness theorem stated below, and next to derive
several results in the following Sections 9--14.

\begin{theorem}\label{clo-th}
Let $D$ be an $\mathcal{L}$-definable subset of $K^{n}$. Then the
canonical projection
$$ \pi: D \times R^{m} \longrightarrow D  $$
is definably closed in the $K$-topology, i.e.\ if $B \subset D
\times R^{m}$ is an $\mathcal{L}$-definable closed subset, so is
its image $\pi(B) \subset D$.
\end{theorem}

\begin{remark}
Not all valued fields $K$ have an angular component map, but it
exists if $K$ has a cross section, which happens whenever $K$ is
$\aleph_{1}$-saturated (cf.~\cite[Chap.~II]{Ch}). Moreover, a
valued field $K$ has an angular component map whenever its residue
field $\Bbbk$ is $\aleph_{1}$-saturated
(cf.~\cite[Corollary~1.6]{Pa2}). In general, unlike for $p$-adic
fields and their finite extensions, adding an angular component
map does strengthen the family of definable sets. Since the
$K$-topology is definable in the language of valued fields, the
closedness theorem is a first order property. Therefore it is
valid over arbitrary Henselian valued fields of equicharacteristic
zero, because it can be proven using saturated elementary
extensions, thus assuming that an angular component map exists.
\end{remark}

Two basic tools applied in this paper are quantifier elimination
for Henselian valued fields (along with preparation cell
decomposition) due to Pas~\cite{Pa1} and relative quantifier
elimination for ordered abelian groups (in a many-sorted language
with imaginary auxiliary sorts) due to
Cluckers--Halupczok~\cite{C-H}. In the case where the ground field
$K$ is of rank one, Theorem~\ref{clo-th} was established in our
paper~\cite[Section~7]{Now2}, where instead we applied simply
quantifier elimination for ordered abelian groups in the
Presburger language. Of course, when $K$ is a locally compact
field, it holds by a routine topological argument.

\vspace{1ex}

As before, our approach relies on the local behavior of definable
functions of one variable and the so-called fiber shrinking, being
a relaxed version of curve selection. Over arbitrary Henselian
valued fields, the former result will be established in Section~5,
and the latter in Section~6. Now, however, in the proofs of fiber
shrinking (Proposition~\ref{FS}) and the closedness theorem
(Theorem~\ref{clo-th}), we also apply relative quantifier
elimination for ordered abelian groups, due to
Cluckers--Halupczok~\cite{C-H}. It will be recalled in Section~7.

\vspace{1ex}

Section~2 contains a version of the implicit function theorem
(Proposition~\ref{implicit}). In the next section, we provide a
version of the Artin--Mazur theorem on algebraic power series
(Proposition~\ref{A-M}). Consequently, every algebraic power
series over $K$ determines a unique continuous function which is
definable in the language of valued fields. Section~4 presents
certain versions of the theorems of Abhyankar--Jung
(Proposition~\ref{A-J}) and Newton-Puiseux (Proposition~\ref{N-P})
for Henselian subalgebras of formal power series which are closed
under power substitution and division by a coordinate, given in
our paper~\cite{Now1} (see also~\cite{P-R}). In Section~5, we use
the foregoing results in analysis of functions of one variable,
definable in the language of Denef--Pas, to establish a theorem on
existence of the limit (Theorem~\ref{limit-th}).

\vspace{1ex}

The closedness theorem will allow us to establish several results
as for instance: piecewise continuity of definable functions
(Section~9), certain non-archimedean versions of curve selection
(Section~10) and of the \L{}ojasiewicz inequality with a direct
consequence, H\"{o}lder continuity of definable functions on
closed bounded subsets of $K^{n}$ (Section~11) as well as
extending hereditarily rational functions (Section~12) and the
theory of regulous functions, sets and sheaves, including
Nullstellensatz and Cartan's theorems A and B (Section~12). Over
rank one valued fields, these results (except piecewise and
H\"{o}lder continuity) were established in our paper~\cite{Now2}.
The theory of hereditarily rational functions on the real and
$p$-adic varieties was developed in the joint paper~\cite{K-N}.
Yet another application of the closedness theorem is the existence
of definable retractions onto closed definable subsets of $K^{n}$
and a definable, non-Archimedean version of the Tietze--Urysohn
extension theorem. These results are established for the algebraic case and 
for Henselian fields with analytic structure in our recent papers~\cite{Now-7,Now-8,Now-9}. 
It is very plausible that they will also hold in the more general case of axiomatically based structures 
on Henselian valued fields.

\vspace{1ex}

The closedness theorem immediately yields five corollaries stated
below. Corollaries~\ref{clo-th-cor-3}
and~\ref{clo-th-cor-4}, enable application of resolution of
singularities and of transformation to a simple normal crossing by
blowing up (cf.~\cite[Chap.~III]{Kol} for references and
relatively short proofs) in much the same way as over locally
compact ground fields.

\vspace{1ex}

\begin{corollary}\label{clo-th-cor-1}
Let $D$ be an $\mathcal{L}$-definable subset of $K^{n}$ and
$\,\mathbb{P}^{m}(K)$ stand for the projective space of dimension
$m$ over $K$. Then the canonical projection
$$ \pi: D \times \mathbb{P}^{m}(K) \longrightarrow D $$
is definably closed.
\end{corollary}

\begin{corollary}\label{clo-th-cor-0}
Let $A$ be a closed $\mathcal{L}$-definable subset of
$\,\mathbb{P}^{m}(K)$ or $R^{m}$. Then every continuous
$\mathcal{L}$-definable map $f: A \to K^{n}$ is definably closed
in the $K$-topology.
\end{corollary}


\begin{corollary}\label{clo-th-cor-2}
Let $\phi_{i}$, $i=0,\ldots,m$, be regular functions on $K^{n}$,
$D$ be an $\mathcal{L}$-definable subset of $K^{n}$ and $\sigma: Y
\longrightarrow K\mathbb{A}^{n}$ the blow-up of the affine space
$K\mathbb{A}^{n}$ with respect to the ideal
$(\phi_{0},\ldots,\phi_{m})$. Then the restriction
$$ \sigma: Y(K) \cap \sigma^{-1}(D) \longrightarrow D $$
is a definably closed quotient map.
\end{corollary}


\begin{proof} Indeed,  $Y(K)$ can be regarded as a closed algebraic subvariety of
$K^{n} \times \mathbb{P}^{m}(K)$ and $\sigma$ as the canonical
projection.
\end{proof}

\begin{corollary}\label{clo-th-cor-3}
Let $X$ be a smooth $K$-variety, $D$ be an $\mathcal{L}$-definable
subset of $X(K)$ and $\sigma: Y \longrightarrow X$ the blow-up
along a smooth center. Then the restriction
$$ \sigma: Y(K) \cap \sigma^{-1}(D) \longrightarrow D $$
is a definably closed quotient map.
\end{corollary}


\begin{corollary}\label{clo-th-cor-4} (Descent property)
Under the assumptions of the above corollary, every continuous
$\mathcal{L}$-definable function
$$ g: Y(K) \cap \sigma^{-1}(D) \longrightarrow K $$
that is constant on the fibers of the blow-up $\sigma$ descends to
a (unique) continuous $\mathcal{L}$-definable function $f: D
\longrightarrow K$.
\end{corollary}


\section{Some versions of the implicit function theorem}

In this section, we give elementary proofs of some versions of the
inverse mapping and implicit function theorems; cf.~the versions
established in the papers \cite[Theorem~7.4]{P-Z},
\cite[Section~9]{G-Pop-Roq}, \cite[Section~4]{Kuhl} and
\cite[Proposition~3.1.4]{G-G-MB}. We begin with a simplest version
(H) of Hensel's lemma in several variables, studied by
Fisher~\cite{Fish}. Given an ideal $\mathfrak{m}$ of a ring $R$,
let $\mathfrak{m}^{\times n}$ stand for the $n$-fold Cartesian
product of $\mathfrak{m}$ and $R^{\times}$ for the set of units of
$R$. The origin $(0,\ldots,0) \in R^{n}$ is denoted by
$\mathbf{0}$.

\vspace{1ex}

\begin{em}
{\bf (H)} Assume that a ring $R$ satisfies Hensel's conditions
(i.e.\  it is linearly topologized, Hausdorff and complete) and
that an ideal $\mathfrak{m}$ of $R$ is closed. Let $f = (f_{1},
\ldots, f_{n})$ be an $n$-tuple of restricted power series $f_{1},
\ldots, f_{n} \in R\{ X \}$, $X = (X_{1},\ldots,X_{n})$, $J$ be
its Jacobian determinant and $a \in R^{n}$. If $f(\mathbf{0}) \in
\mathfrak{m}^{\times n}$ and $J(\mathbf{0}) \in R^{\times}$, then
there is a unique $a \in \mathfrak{m}^{\times n}$ such that $f(a)
= \mathbf{0}$.
\end{em}

\begin{proposition}\label{H-1}
Under the above assumptions, $f$ induces a bijection
$$ \mathfrak{m}^{\times n} \ni x \longrightarrow f(x) \in \mathfrak{m}^{\times n} $$
 of $\mathfrak{m}^{\times n}$ onto itself.
\end{proposition}

\begin{proof}
For any $y \in\mathfrak{m}^{\times n}$, apply condition (H) to the
restricted power series $f(X) - y$.
\end{proof}

If, moreover, the pair $(R,\mathfrak{m})$ satisfies Hensel's
conditions (i.e.\ every element of $\mathfrak{m}$ is topologically
nilpotent), then condition (H) holds by \cite[Chap.~III, \S
4.5]{Bour}.

\begin{remark}\label{rem-char}
Henselian local rings can be characterized both by the classical
Hensel lemma and by condition (H): a local ring $(R,\mathfrak{m})$
is Henselian iff $(R,\mathfrak{m})$ with the discrete topology
satisfies condition (H) (cf.~~\cite[Proposition~2]{Fish}).
\end{remark}

Now consider a Henselian local ring $(R,\mathfrak{m})$. Let $f =
(f_{1}, \ldots, f_{n})$ be an $n$-tuple of polynomials $f_{1},
\ldots, f_{n} \in R[ X ]$, $X = (X_{1},\ldots,X_{n})$ and $J$ be
its Jacobian determinant.

\begin{corollary}\label{H-2}
Suppose that $f(\mathbf{0}) \in \mathfrak{m}^{\times n}$ and
$J(\mathbf{0}) \in R^{\times}$. Then $f$ is a homeomorphism of
$\mathfrak{m}^{\times n}$ onto itself in the $\mathfrak{m}$-adic
topology. If, in addition, $R$ is a Henselian valued ring with
maximal ideal $\mathfrak{m}$, then $f$ is a homeomorphism of
$\mathfrak{m}^{\times n}$ onto itself in the valuation topology.
\end{corollary}

\begin{proof}
Obviously, $J(a) \in R^{\times}$ for every $a \in
\mathfrak{m}^{\times n}$. Let $\mathcal{M}$ be the jacobian matrix
of $f$. Then
$$ f(a + x) - f(a) = \mathcal{M}(a) \cdot x + g(x) = \mathcal{M}(a)
   \cdot (x + \mathcal{M}(a)^{-1} \cdot g(x)) $$
for an $n$-tuple $g = (g_{1},\ldots,g_{n})$ of polynomials
$g_{1},\ldots,g_{n} \in (X)^{2} R[X]$. Hence the assertion follows
easily.
\end{proof}

The proposition below is a version of the inverse mapping theorem.

\begin{proposition}\label{H-3}
If $f(\mathbf{0}) = \mathbf{0}$ and $e :=J(\mathbf{0}) \neq 0$,
then $f$ is an open embedding of $e \cdot \mathfrak{m}^{\times n}$
onto $e^{2} \cdot \mathfrak{m}^{\times n}$.
\end{proposition}

\begin{proof}
Let $\mathcal{N}$ be the adjugate of the matrix
$\mathcal{M}(\mathbf{0})$ and $y = e^{2}b$ with $b \in
\mathfrak{m}^{\times n}$. Since
$$ f(eX) = e \cdot \mathcal{M}(\mathbf{0}) \cdot X + e^{2} g(X) $$
for an $n$-tuple $g = (g_{1},\ldots,g_{n})$ of polynomials
$g_{1},\ldots,g_{n} \in (X)^{2} R[X]$, we get the equivalences
$$ f(eX) = y \ \Leftrightarrow \ f(eX) - y = \mathbf{0} \ \Leftrightarrow \
   e \cdot \mathcal{M}(\mathbf{0}) \cdot (X + \mathcal{N}g(X) - \mathcal{N}b) = \mathbf{0}. $$
Applying Corollary~\ref{H-2} to the map $h(X) := X +
\mathcal{N}g(X)$, we get
$$ f^{-1}(y) = ex \ \Leftrightarrow \ x = h^{-1}(\mathcal{N}b) \
   \ \text{and} \ \ f^{-1}(y) = e h^{-1}(\mathcal{N} \cdot y/e^{2}). $$
This finishes the proof.
\end{proof}

\vspace{1ex}

Further, let $0 \leq r < n$, $p = (p_{r+1}, \ldots, p_{n})$ be an
$(n-r)$-tuple of polynomials $p_{r+1},\ldots,p_{n} \in R[X]$, $X =
(X_{1},\ldots,X_{n})$,  and
$$ J := \frac{\partial(p_{r+1}, \ldots,
   p_{n})}{\partial(X_{r+1},\ldots,X_{n})}, \ \ e := J(\mathbf{0}). $$
Suppose that
$$ \mathbf{0} \in V := \{ x \in R^{n}: p_{r+1}(x) = \ldots = p_{n}(x) =
   0 \}. $$
In a similar fashion as above, we can establish the following
version of the implicit function theorem.

\begin{proposition}\label{implicit}
If $e \neq 0$, then there exists a unique continuous map
$$ \phi: (e^{2} \cdot \mathfrak{m})^{\times r} \longrightarrow
   (e \cdot \mathfrak{m})^{\times (n-r)} $$
which is definable in the language of valued fields and such that
$\phi(0)=0$ and the graph map
$$ (e^{2} \cdot \mathfrak{m})^{\times r} \ni u \longrightarrow (u,\phi(u)) \in
   (e^{2} \cdot \mathfrak{m})^{\times r} \times
   (e \cdot \mathfrak{m})^{\times (n-r)} $$
is an open embedding into the zero locus $V$ of the polynomials
$p$ and, more precisely, onto
$$ V \cap \left[ (e^{2} \cdot \mathfrak{m})^{\times r} \times
   (e \cdot \mathfrak{m})^{\times (n-r)} \right]. $$
\end{proposition}

\begin{proof}
Put $f(X) := (X_{1},\ldots,X_{r},p(X))$; of course, the jacobian
determinant of $f$ at $\mathbf{0} \in R^{n}$ is equal to $e$. Keep
the notation from the proof of Proposition~\ref{H-3}, take any $b
\in e^{2} \cdot \mathfrak{m}^{\times r}$ and put $y := (e^{2}b,0)
\in R^{n}$. Then we have the equivalences
$$ f(eX)=y \ \Leftrightarrow \ f(eX) - y= \mathbf{0} \
   \Leftrightarrow \ e \mathcal{M}(\mathbf{0}) \cdot
   (X + \mathcal{N} g(X) - \mathcal{N} \cdot (b,0)) = \mathbf{0}. $$
Applying Corollary~\ref{H-2} to the map $h(X) := X +
\mathcal{N}g(X)$, we get
$$ f^{-1}(y) = ex \ \Leftrightarrow \ x = h^{-1}(\mathcal{N} \cdot (b,0)) \
   \ \text{and} \ \ f^{-1}(y) = e h^{-1}(\mathcal{N} \cdot y/e^{2}). $$
Therefore the function
$$ \phi(u) := e h^{-1}(\mathcal{N}\cdot (u,0)/e^{2}) $$
is the one we are looking for.
\end{proof}

\section{Density property and a version of the Artin--Mazur
theorem over Henselian valued fields}

We say that a topological field $K$ satisfies the {\it density
property} (cf.~\cite{K-N,Now2}) if the following equivalent
conditions hold.
\begin{enumerate}
\item If $X$ is a smooth, irreducible $K$-variety and
    $\emptyset\neq U\subset X$ is a Zariski open subset,
then $U(K)$ is dense in $X(K)$ in the $K$-topology.
\item If $C$ is a smooth, irreducible $K$-curve and
    $\emptyset\neq U$ is a Zariski open subset, then $U(K)$ is
    dense in
$C(K)$ in the $K$-topology.
\item If $C$ is a smooth, irreducible $K$-curve, then $C(K)$
    has no isolated points.
\end{enumerate}
(This property is indispensable for ensuring reasonable
topological and geometric properties of algebraic subsets of
$K^{n}$; see~\cite{Now2} for the case where the ground field $K$
is a Henselian rank one valued field.) The density property of
Henselian non-trivially valued fields follows immediately from
Proposition~\ref{implicit} and the Jacobian criterion for
smoothness (see e.g.~\cite[Theorem~16.19]{Eis}), recalled below
for the reader's convenience.

\begin{theorem}\label{smooth}
Let $I = (p_{1}, \ldots, p_{s}) \subset K[X]$, $X =
(X_{1},\ldots,X_{n})$ be an ideal, $A := K[X]/I$ and $V :=
\mathrm{Spec}\, (A)$. Suppose the origin $\mathbf{0} \in K^{n}$
lies in $V$ (equivalently, $I \subset (X)K[X]$) and $V$ is of
dimension $r$ at $\mathbf{0}$. Then the Jacobian matrix
$$ \mathcal{M} := \left[
   \frac{\partial p_{i}}{\partial X_{j}}(\mathbf{0}): \: i=1,\ldots,s, \: j=1,\ldots,n
   \right] $$
has rank $\leq (n-r)$ and $V$ is smooth at $\mathbf{0}$ iff
$\mathcal{M}$ has exactly rank $(n-r)$. Furthermore, if $V$ is
smooth at $\mathbf{0}$ and
$$ \mathcal{J} := \frac{\partial (p_{r+1},\ldots,p_{n})}{\partial (X_{r+1},\ldots,X_{n})}
   (\mathbf{0}) = \det \left[ \frac{\partial p_{i}}{\partial X_{j}} (\mathbf{0}):
   \: i,j=r+1,\ldots,n \right] \neq 0, $$
then $p_{r+1},\ldots,p_{n}$ generate the localization $I \cdot
K[X]_{(X_{1},\ldots,X_{n})}$ of the ideal $I$ with respect to the
maximal ideal $(X_{1},\ldots,X_{n})$.
\end{theorem}

\begin{remark}\label{etale}
Under the above assumptions, consider the completion
$$ \widehat{A} = K[[ X ]]/ I \cdot K[[ X ]] $$
of $A$ in the $(X)$-adic topology. If $\mathcal{J} \neq 0$, it
follows from the implicit function theorem for formal power series
that there are unique power series
$$ \phi_{r+1},\ldots,\phi_{n} \in (X_{1},\ldots,X_{r}) \cdot K[[X_{1},\ldots,X_{r}]] $$
such that
$$ p_{i}(X_{1},\ldots,X_{r},\phi_{r+1}(X_{1},\ldots,X_{r}), \ldots,
   \phi_{n}(X_{1},\ldots,X_{r})) = 0 $$
for $i=r+1,\ldots,n$. Therefore the homomorphism
$$ \widehat{\alpha}: \widehat{A} \longrightarrow K[[X_{1},\ldots,X_{r}]], \ \
   X_{j} \mapsto X_{j}, \ X_{k} \mapsto \phi_{k}(X_{1},\ldots,X_{r}), $$
for $j=1,\ldots,r$ and $k=r+1,\ldots,n$, is an isomorphism.

Conversely, suppose that $\widehat{\alpha}$ is an isomorphism;
this means that the projection from $V$ onto $\mathrm{Spec}\,
K[X_{1},\ldots,X_{r}]$ is etale at $\mathbf{0}$. Then the local
rings $A$ and $\widehat{A}$ are regular and, moreover, it is easy
to check that the determinant $\mathcal{J} \neq 0$ does not vanish
after perhaps renumbering the polynomials $p_{i}(X)$.
\end{remark}

We say that a formal power series $\phi \in K[[X]]$,
$X=(X_{1},\ldots,X_{n})$, is algebraic if it is algebraic over
$K[X]$. The kernel of the homomorphism of $K$-algebras
$$ \sigma: K[X,T] \longrightarrow K[[X]], \ \
   X_{1} \mapsto X_{1}, \, \ldots \, , X_{n} \mapsto X_{n}, \, T \mapsto \phi(X), $$
is, of course, a principal prime ideal:
$$ \mathrm{ker}\, \sigma = (p) \subset K[X,T], $$
where $p \in K[X,T]$ is a unique (up to a constant factor)
irreducible polynomial, called an \emph{irreducible polynomial} of
$\phi$.

\vspace{1ex}

We now state a version of the Artin--Mazur theorem
(cf.~\cite{AM,BCR} for the classical versions).

\begin{proposition}\label{A-M}
Let $\phi \in (X) K[[X]]$ be an algebraic formal power series.
Then there exist polynomials
$$ p_{1},\ldots ,p_{r} \in K[X,Y], \ \ Y=(Y_{1},\ldots,Y_{r}), $$
and formal power series $\phi_{2},\ldots,\phi_{r} \in K[[X]]$ such
that
$$ e := \frac{\partial (p_{1},\ldots ,p_{r})}{\partial (Y_{1},\ldots ,Y_{r})}
   (\mathbf{0}) = \det \left[ \frac{\partial p_{i}}{\partial Y_{j}} (\mathbf{0}):
   \: i,j=1,\ldots,r \right] \neq 0, $$
and
$$ p_{i}(X_{1},\ldots,X_{n},\phi_{1}(X), \ldots, \phi_{r}(X)) = 0, \ \ i=1,\ldots,r, $$
where $\phi_{1} := \phi$.
\end{proposition}

\begin{proof}
Let $p_{1}(X,Y_{1})$ be an irreducible polynomial of $\phi_{1}$.
Then the integral closure $B$ of $A := K[X,Y_{1}]/(p_{1})$ is a
finite $A$-module and thus is of the form
$$ B = K[X,Y]/(p_{1},\ldots,p_{s}), \ \ Y=(Y_{1},\ldots,Y_{r}), $$
where $p_{1},\ldots,p_{s} \in K[X,Y]$. Obviously, $A$ and $B$ are
of dimension $n$, and the induced embedding $\alpha: A \to K[[X]]$
extends to an embedding $\beta: B \to K[[X]]$. Put
$$ \phi_{k} := \beta(Y_{k}) \in K[[X]], \ \ k=1,\ldots,r. $$
Substituting $Y_{k} - \phi_{k}(0)$ for $Y_{k}$, we may assume that
$\phi_{k}(0) = 0$ for all $k=1,\ldots,r$. Hence $p_{i}(\mathbf{0})
=0$ for all $i=1,\ldots,s$.

The completion $\widehat{B}$ of $B$ in the $(X,Y)$-adic topology
is a local ring of dimension $n$, and the induced homomorphism
$$ \widehat{\beta}: \widehat{B} =
   K[[X,Y]]/(p_{1},\ldots,p_{s}) \longrightarrow K[[X]] $$
is, of course, surjective. But, by the Zariski main theorem
(cf.~\cite[Chap.~VIII, \S~13, Theorem~32]{Z-S}), $\widehat{B}$ is
a normal domain. Comparison of dimensions shows that
$\widehat{\beta}$ is an isomorphism. Now, it follows from
Remark~\ref{etale} that the determinant $e \neq 0$ does not vanish
after perhaps renumbering the polynomials $p_{i}(X)$. This
finishes the proof.
\end{proof}

Propositions~\ref{A-M} and~\ref{implicit} immediately yield the
following

\begin{corollary}
Let $\phi \in (X) K[[X]]$ be an algebraic power series with
irreducible polynomial $p(X,T) \in K[X,T]$. Then there is an $a
\in K$, $a \neq 0$, and a unique continuous function
$$ \widetilde{\phi}: a \cdot R^{n} \longrightarrow K $$
corresponding to $\phi$, which is definable in the language of
valued fields and such that $\widetilde{\phi}(0) = 0$ and
$p(x,\widetilde{\phi}(x)) =0$ for all $x \in a \cdot R^{n}$.
\hspace*{\fill}$\Box$
\end{corollary}

For simplicity, we shall denote the induced continuous function by
the same letter $\phi$. This abuse of notation will not lead to
confusion in general.

\begin{remark}
Clearly, the ring $K[[X]]_{alg}$ of algebraic power series is the
henselization of the local ring $K[X]_{(X)}$ of regular functions.
Therefore the implicit functions
$\phi_{r+1}(u),\ldots,\phi_{n}(u)$ from Proposition~\ref{implicit}
correspond to unique algebraic power series
$$ \phi_{r+1}(X_{1},\ldots,X_{r}),\ldots,\phi_{n}(X_{1},\ldots,X_{r}) $$
without constant term. In fact, one can deduce by means of the
classical version of the implicit function theorem for restricted
power series (cf.~\cite[Chap.~III, \S 4.5]{Bour} or~\cite{Fish})
that $\phi_{r+1},\ldots,\phi_{n}$ are of the form
$$ \phi_{k}(X_{1},\ldots,X_{r}) = e \cdot \omega_{k} (X_{1}/e^{2}, \ldots,
   X_{r}/e^{2}), \ \ k= r+1,\ldots, n, $$
where $\omega_{k}(X_{1},\ldots, X_{r}) \in
R[[X_{1},\ldots,X_{r}]]$ and $e \in R$.
\end{remark}

\section{The Newton--Puiseux and Abhyankar--Jung Theorems}

Here we are going to provide a version of the Newton--Puiseux
theorem, which will be used in analysis of definable functions of
one variable in the next section.

\vspace{1ex}

We call a polynomial
$$ f(X;T)= T^{s} + a_{s-1}(X)T^{n-1} + \cdots + a_{0}(X) \in K[[X]][T],
$$
$X=(X_{1},\ldots,X_{s})$, quasiordinary if its discriminant $D(X)$
is a normal crossing:
$$ D(X) = X^{\alpha} \cdot u(X) \ \ \ \mbox{ with } \ \ \alpha \in
   \mathbb{N}^{s}, \ u(X) \in k[[X]], \ u(0) \neq 0. $$

Let $K$ be an algebraically closed field of characteristic zero.
Consider a henselian $K[X]$-subalgebra $K \langle X \rangle$ of
the formal power series ring $K[[X]]$ which is closed under
reciprocal (whence it is a local ring), power substitution and
division by a coordinate. For positive integers
$r_{1},\ldots,r_{n}$ put
$$ K \langle X_{1}^{1/r_{1}},\ldots, X_{n}^{1/r_{n}} \rangle := \left\{
   a( X_{1}^{1/r_{1}},\ldots, X_{n}^{1/r_{n}}): a(X) \in K \langle X
   \rangle \right\}; $$
when $r_{1}= \ldots =r_{m}=r$, we denote the above algebra by $K
\langle X^{1/r} \rangle$.

\vspace{1ex}

In our paper~\cite{Now1} (see also~\cite{P-R}), we established a
version of the Abhyankar--Jung theorem recalled below. This
axiomatic approach to that theorem was given for the first time in
our preprint~\cite{Now0}.

\begin{proposition}\label{A-J}
Under the above assumptions, every quasi\-ordinary polynomial
$$ f(X;T)= T^{s} + a_{s-1}(X)T^{s-1} + \cdots + a_{0}(X)
   \in K \langle X \rangle[T] $$
has all its roots in $K \langle X^{1/r} \rangle$ for some $r \in
\mathbb{N}$; actually, one can take $r = s!$.
\end{proposition}

A particular case is the following version of the Newton-Puiseux
theorem.

\begin{corollary}\label{N-P}
Let $X$ denote one variable. Every polynomial
$$ f(X;T)= T^{s} + a_{s-1}(X)T^{s-1} + \cdots + a_{0}(X)
   \in K \langle X \rangle[T] $$
has all its roots in $K \langle X^{1/r} \rangle$ for some $r \in
\mathbb{N}$; one can take $r = s!$. Equivalently, the polynomial
$f(X^{r},T)$ splits into $T$-linear factors. If $f(X,T)$ is
irreducible, then $r=s$ will do and
$$ f(X^{s},T) = \prod_{i=1}^{s} \, (T - \phi(\epsilon^{i}X)), $$
where $\phi(X) \in K \langle X \rangle$ and $\epsilon$ is a
primitive root of unity.
\end{corollary}

\begin{remark}\label{puis}
Since the proof of these theorems is of finitary character, it is
easy to check that if the ground field $K$ of characteristic zero
is not algebraically closed, they remain valid for the Henselian
subalgebra $\overline{K} \otimes_{K} K \langle X \rangle$ of
$\overline{K}[[ X ]]$, where $\overline{K}$ denotes the algebraic
closure of $K$.
\end{remark}

\vspace{1ex}

The ring $K[[X]]_{alg}$ of algebraic power series is a local
Henselian ring closed under power substitutions and division by a
coordinate. Thus the above results apply to the algebra $K \langle
X \rangle = K[[X]]_{alg}$.

\section{Definable functions of one variable}

At this stage, we can readily to proceed with analysis of
definable functions of one variable over arbitrary Henselian
valued fields of equicharacteristic zero. We wish to establish a
general version of the theorem on existence of the limit stated
below. It was proven in~\cite[Proposition~5.2]{Now2} over rank one
valued fields. Now the language $\mathcal{L}$ under consideration
is the three-sorted language of Denef--Pas.

\begin{theorem}\label{limit-th} (Existence of the limit)
Let $f:A \to K$ be an $\mathcal{L}$-definable function on a subset
$A$ of $K$ and suppose $0$ is an accumulation point of $A$. Then
there is a finite partition of $A$ into $\mathcal{L}$-definable
sets $A_{1},\ldots,A_{r}$ and points $w_{1}\ldots,w_{r} \in
\mathbb{P}^{1}(K)$ such that
$$ \lim_{x \rightarrow 0}\, f|A_{i}\, (x) = w_{i} \ \ \ \text{for} \ \
   i=1,\ldots,r. $$
Moreover, there is a neighborhood $U$ of $0$ such that each
definable set
$$ \{ (v(x), v(f(x))): \; x \in (A_{i} \cap U) \setminus \{0 \} \}
   \subset \Gamma \times (\Gamma \cup \ \{
   \infty \}),  \ i=1,\ldots,r, $$
is contained in an affine line with rational slope
$$ q \cdot l = p_{i} \cdot k + \beta_{i}, \ \ i=1,\ldots,r, $$
with $p_{i},q \in \mathbb{Z}$, $q>0$, $\beta_{i} \in \Gamma$, or
in\/ $\Gamma \times \{ \infty \}$.
\end{theorem}

\begin{proof}
Having the Newton--Puiseux theorem for algebraic power series at
hand, we can repeat mutatis mutandis the proof from loc.~cit.\ as
briefly outlined below. In that paper, the field $L$ is the
completion of the algebraic closure $\overline{K}$ of the ground
field $K$. Here, in view of Corollary~\ref{puis}, the $K$-algebras
$L\{ X \}$ and $\widehat{K}\{ X \}$ should be just replaced with
$\overline{K} \otimes_{K} K[[X]]_{alg}$ and $K[[X]]_{alg}$,
respectively. Then the reasonings follow almost verbatim. Note
also that Lemma~5.1 (to the effect that $K$ is a closed subspace
of $\overline{K}$) holds true for arbitrary Henselian valued
fields of equicharacteristic zero. This follows directly from that
the field $K$ is algebraically maximal (as it is Henselian and
finitely ramified; see e.g.~\cite[Chap.~4]{E-Pre}).
\end{proof}

We conclude with the following comment. The above proposition
along with the technique of fiber shrinking
from~\cite[Section~6]{Now2} were two basic tools in the proof of
the closedness theorem~\cite[Theorem~3.1]{Now2} over Henselian
rank one valued fields, which plays an important role in Henselian
geometry.

\vspace{1ex}


\section{Fiber shrinking}

Consider a Henselian valued field $K$ of equicharacteristic zero
along with the three-sorted language $\mathcal{L}$ of Denef--Pas.
In this section, we remind the reader the concept of fiber
shrinking introduced in our paper~\cite[Section~6]{Now2}.

\vspace{1ex}

Let $A$ be an $\mathcal{L}$-definable subset of $K^{n}$ with
accumulation point
$$ a = (a_{1},\ldots,a_{n}) \in K^{n} $$
and $E$ an $\mathcal{L}$-definable subset of $K$ with accumulation
point $a_{1}$. We call an $\mathcal{L}$-definable family of sets
$$ \Phi = \bigcup_{t \in E} \ \{ t \} \times \Phi_{t} \subset A $$
an $\mathcal{L}$-definable $x_{1}$-fiber shrinking for the set $A$
at $a$ if
$$ \lim_{t \rightarrow a_{1}} \, \Phi_{t} = (a_{2},\ldots,a_{n}),
$$
i.e.\ for any neighborhood $U$ of $(a_{2},\ldots,a_{n}) \in
K^{n-1}$, there is a neighborhood $V$ of $a_{1} \in K$ such that
$\emptyset \neq \Phi_{t} \subset U$ for every $t \in V \cap E$, $t
\neq a_{1}$. When $n=1$, $A$ is itself a fiber shrinking for the
subset $A$ of $K$ at an accumulation point $a \in K$.


\begin{proposition}\label{FS} (Fiber shrinking)
Every $\mathcal{L}$-definable subset $A$ of $K^{n}$ with
accumulation point $a \in K^{n}$ has, after a permutation of the
coordinates, an $\mathcal{L}$-definable $x_{1}$-fiber shrinking at
$a$.
\end{proposition}

In the case where the ground field $K$ is of rank one, the proof
of Proposition~\ref{FS} was given in~\cite[Section~6]{Now2}. In
the general case, it can be repeated verbatim once we demonstrate
the following result on definable subsets in the value group sort
$\Gamma$.

\begin{lemma}\label{line}
Let $\Gamma$ be an ordered abelian group and $P$ be a definable
subset of $\Gamma^{n}$. Suppose that $(\infty,\ldots,\infty)$ is
an accumulation point of $P$, i.e.\ for any $\delta \in \Gamma$
the set
$$ \{ x \in P: x_{1} > \delta, \ldots, x_{n} > \delta \} \neq \emptyset $$
is non-empty. Then there is an affine semi-line
$$ L = \{ (r_{1}k + \gamma_{1},\ldots,r_{n}k + \gamma_{n}): \, k
   \in \Gamma, \ k \geq 0 \} \ \ \ \text{with} \ \ r_{1},\ldots,r_{n} \in \mathbb{N}, $$
passing through a point $\gamma = (\gamma_{1},\ldots,\gamma_{n})
\in P$ and such that $(\infty,\ldots,\infty)$ is an accumulation
point of the intersection $P \cap L$ too.
\end{lemma}

In~\cite[Section~6]{Now2}, Lemma~\ref{line} was established for
archimedean groups by means of quantifier elimination in the
Presburger language. Now, in the general case, it follows in a
similar fashion by means of relative quantifier elimination for
ordered abelian groups in the language $\mathcal{L}_{qe}$ due to
Cluckers--Halupczok~\cite{C-H}, outlined in the next section.
Indeed, applying Theorem~\ref{RQE} along with Remarks~\ref{Rem1}
and~\ref{Rem2}), it is not difficult to see that the parametrized
congruence conditions which occur in the description of the set
$P$ are not an essential obstacle to finding the line $L$ we are
looking for. Therefore the lemma reduces, likewise as it was
in~\cite[Section~6]{Now2}, to a problem of semi-linear geometry.

\section{Quantifier elimination for ordered abelian
groups}

It is well known that archimedean ordered abelian groups admit
quantifier elimination in the Presburger language. Much more
complicated are quantifier elimination results for non-archimedean
groups (especially those with infinite rank), going back as far as
Gurevich~\cite{Gur}. He established a transfer of sentences from
ordered abelian groups to so-called coloured chains (i.e.\
linearly ordered sets with additional unary predicates), enhanced
later to allow arbitrary formulas. This was done in his doctoral
dissertation "The decision problem for some algebraic theories"
(Sverdlovsk, 1968), and by Schmitt in his habilitation
dissertation "Model theory of ordered abelian groups" (Heidelberg,
1982); see also the paper~\cite{Sch}. Such a transfer is a kind of
relative quantifier elimination, which allows
Gurevich--Schmitt~\cite{G-S}, in their study of the NIP property,
to lift model theoretic properties from ordered sets to ordered
abelian groups or, in other words, to transform statements on
ordered abelian groups into those on coloured chains.

\vspace{1ex}

Instead Cluckers--Halupczok~\cite{C-H} introduce a suitable
many-sorted language $\mathcal{L}_{qe}$ with main group sort
$\Gamma$ and auxiliary imaginary sorts (with canonical parameters
for some definable families of convex subgroups) which carry the
structure of a linearly ordered set with some additional unary
predicates. They provide quantifier elimination relative to the
auxiliary sorts, where each definable set in the group sort is a
union of a family of quantifier free definable sets with parameter
running a definable (with quantifiers) set of the auxiliary sorts.

\vspace{1ex}

Fortunately, sometimes it is possible to directly deduce
information about ordered abelian groups without any deeper
knowledge of the auxiliary sorts. For instance, this may be
illustrated by their theorem on piecewise linearity of definable
functions~\cite[Corollary~1.10]{C-H} as well as by
Proposition~\ref{line} and application of quantifier elimination
in the proof of the closedness theorem in Section~4.

\vspace{1ex}

Now we briefly recall the language $\mathcal{L}_{qe}$ taking care
of points essential for our applications. The main group sort
$\Gamma$ is with the constant $0$, the binary function $+$ and the
unary function $-$. The collection $\mathcal{A}$ of auxiliary
sorts consists of certain imaginary sorts:
$$ \mathcal{A} := \{ \mathcal{S}_{p}, \mathcal{T}_{p},
   \mathcal{T}^{+}_{p}: p \in \mathbb{P} \}; $$
here $\mathbb{P}$ stands for the set of prime numbers. By abuse of
notation, $\mathcal{A}$ will also denote the union of the
auxiliary sorts. In this section, we denote $\Gamma$-sort
variables by $x,y,z,\ldots$ and auxiliary sorts variables by
$\eta, \theta, \zeta, \ldots$.

\vspace{1ex}

Further, the language $\mathcal{L}_{qe}$ consists of some unary
predicates on $\mathcal{S}_{p}$, $p \in \mathbb{P}$, some binary
order relations on $\mathcal{A}$, a ternary relation
$$ x \equiv_{m,\alpha}^{m'} y \ \ \text{on} \ \
   \Gamma \times \Gamma \times \mathcal{S}_{p} \ \
   \text{for each} \ \ p \in \mathbb{P}, \ m,m' \in \mathbb{N}, $$
and finally predicates for the ternary relations
$$ x \diamond_{\alpha} y + k_{\alpha} \ \ \text{on} \ \ \Gamma \times \Gamma
   \times \mathcal{A}, $$
where $\diamond \in \{ =, <, \equiv_{m} \}$, $m \in \mathbb{N}$,
$k \in \mathbb{Z}$ and $\alpha$ is the third operand running any
of the auxiliary sorts $\mathcal{A}$.

\vspace{1ex}

We now explain the meaning of the above ternary relations, which
are defined by means of certain definable convex subgroups
$\Gamma_{\alpha}$ and $\Gamma_{\alpha}^{m'}$ of $\Gamma$ with
$\alpha \in \mathcal{A}$ and $m' \in \mathbb{N}$. Namely we write
$$ x \equiv_{m,\alpha}^{m'} y \ \ \text{iff} \ \  x-y \in
   \Gamma_{\alpha}^{m'} + m\Gamma. $$
Further, let $1_{\alpha}$ denote the minimal positive element of
$\Gamma/\Gamma_{\alpha}$ if $\Gamma/\Gamma_{\alpha}$ is discrete
and $1_{\alpha} :=0$ otherwise, and set $k_{\alpha} := k \cdot
1_{\alpha}$ for all $k \in \mathbb{Z}$. By definition we write
$$ x \diamond_{\alpha} y + k_{\alpha} \ \ \ \text{iff} \ \ \
   x \, (\bmod \, \Gamma_{\alpha}) \diamond y \ (\bmod \, \Gamma_{\alpha}) +
   k_{\alpha}. $$
(Thus the language $\mathcal{L}_{qe}$ incorporates the Presburger
language on all quotients $\Gamma/\Gamma_{\alpha}$.) Note also
that the ordinary predicates $<$ and $\equiv_{m}$ on $\Gamma$ are
$\Gamma$-quantifier-free definable in the language
$\mathcal{L}_{qe}$.

\vspace{1ex}

Now we can readily formulate quantifier elimination relative to
the auxiliary sorts (\cite[Theorem~1.8]{C-H}).

\begin{theorem}\label{RQE}
In the theory $T$ of ordered abelian groups, each
$\mathcal{L}_{qe}$-formula $\phi(\bar{x},\bar{\eta})$ is
equivalent to an $\mathcal{L}_{qe}$-formula
$\psi(\bar{x},\bar{\eta})$ in family union form, i.e.\
$$ \psi(\bar{x},\bar{\eta}) = \bigvee_{i=1}^{k} \; \exists \,
   \bar{\theta} \: \left[ \chi_{i}(\bar{\eta},\bar{\theta})
   \wedge \omega_{i}(\bar{x},\bar{\theta}) \right], $$
where $\bar{\theta}$ are $\mathcal{A}$-variables, the formulas
$\chi_{i}(\bar{\eta},\bar{\theta})$ live purely in the auxiliary
sorts $\mathcal{A}$, each $\omega_{i}(\bar{x},\bar{\theta})$ is a
conjunction of literals (i.e.\ atomic or negated atomic formulas)
and $T$ implies that the $\mathcal{L}_{qe}(\mathcal{A})$-formulas
$$ \{ \chi_{i}(\bar{\eta},\bar{\alpha}) \wedge
   \omega_{i}(\bar{x},\bar{\alpha}): \; i=1,\ldots,k, \
   \bar{\alpha} \in \mathcal{A} \} $$
are pairwise inconsistent.
\end{theorem}

\begin{remark}\label{Rem1}
The sets definable (or, definable with parameters) in the main
group sort $\Gamma$ resemble to some extent the sets which are
definable in the Presburger language. Indeed, the atomic formulas
involved in the formulas $\omega_{i}(\bar{x},\bar{\theta})$ are of
the form
$$ t(\bar{x}) \diamond_{\theta_{j}} k_{\theta_{j}}, $$
where $t(\bar{x})$ is a $\mathbb{Z}$-linear combination
(respectively, a $\mathbb{Z}$-linear combination plus an element
of $\Gamma$) , the predicates
$$ \diamond \in \{ =, <, \equiv_{m}, \equiv_{m}^{m'} \} \ \ \text{with some}
   \ \ m,m' \in \mathbb{N}, $$
$\theta_{j}$ is one of the entries of $\bar{\theta}$ and $k \in
\mathbb{Z}$; here $k=0$ if $\diamond$ is $\equiv_{m}^{m'}$.
Clearly, while linear equalities and inequalities define
polyhedra, congruence conditions define sets which consist of
entire cosets of $m\Gamma$ for finitely many $m \in \mathbb{N}$.
\end{remark}

\begin{remark}\label{Rem2}
Note also that the sets given by atomic formulas $t(\bar{x})
\diamond_{\theta_{j}} k_{\theta_{j}}$ consist of entire cosets of
the subgroups $\Gamma_{\theta_{j}}$. Therefore, the union of those
subgroups $\Gamma_{\theta_{j}}$ which essentially occur in a
formula in family union form, describing a proper subset of
$\Gamma^{n}$, is not cofinal with $\Gamma$. This observation is
often useful as, for instance, in the proofs of fiber shrinking
and Theorem~\ref{clo-th}.
\end{remark}

\section{Proof of the closedness theorem}

In the proof of Theorem~\ref{clo-th}, we shall generally follow
the ideas from our previous paper~\cite[Section~7]{Now2}. We must
show that if $B$ is an $\mathcal{L}$-definable subset of $D \times
(K^{\circ})^{n}$ and a point $a$ lies in the closure of $A :=
\pi(B)$, then there is a point $b$ in the closure of $B$ such that
$\pi(b)=a$. Again, the proof reduces easily to the case $m=1$ and
next, by means of fiber shrinking (Proposition~\ref{FS}), to the
case $n=1$. We may obviously assume that $a = 0 \not \in A$.

\vspace{1ex}

Whereas in the paper~\cite{Now2} preparation cell decomposition
(due to Pas; see \cite[Theorem~3.2]{Pa1}
and~\cite[Theorem~2.4]{Now2}) was combined with quantifier
elimination in the $\Gamma$ sort in the Presburger language, here
it is combined with relative quantifier elimination in the
language $\mathcal{L}_{qe}$ considered in Section~7. In a similar
manner as in~\cite{Now2}, we can now assume that $B$ is a subset
$F$ of a cell $C$ of the form presented below. Let
$$ a(x,\xi),b(x,\xi),c(x,\xi): \, D \longrightarrow K $$
be three $\mathcal{L}$-definable functions on an
$\mathcal{L}$-definable subset $D$ of $K^{2} \times \Bbbk^{m}$ and
let $\nu \in \mathbb{N}$ is a positive integer. For each $\xi \in
\Bbbk^{m}$ set
$$ C(\xi) := \left\{ \rule{0em}{3ex} (x,y) \in K^{n}_{x} \times K_{y}:
   \ (x,\xi) \in D, \right. $$
$$ \left. \rule{0em}{3ex} v(a(x,\xi)) \lhd_{1} v((y-c(x,\xi))^{\nu}) \lhd_{2} v(b(x,\xi)), \
   \overline{ac} (y-c(x,\xi)) = \xi_{1} \right\}, $$
where $\lhd_{1},\lhd_{2}$ stand for $<, \leq$ or no condition in
any occurrence. A cell $C$ is by definition a disjoint union of
the fibres $C(\xi)$. The subset $F$ of $C$ is a union of fibers
$F(\xi)$ of the form
$$ F(\xi) := \left\{ \rule{0em}{4ex} (x,y) \in C(\xi): \ \exists \
   \bar{\theta} \ \chi(\bar{\theta}) \ \wedge\ \right. $$
$$ \bigwedge_{i \in I_{a}} \
   v(a_{i}(x,\xi)) \lhd_{1,\theta_{j_{i}}} v((y - c(x,\xi))^{\nu_{i}}), \
   \bigwedge_{i \in I_{b}} \
   v((y - c(x,\xi))^{\nu_{i}}) \lhd_{2,\theta_{j_{i}}} v(b_{i}(x,\xi)) $$
$$  \left. \wedge \ \bigwedge_{i \in I_{f}} \
    v((y - c(x,\xi))^{\nu_{i}}) \diamond_{\theta_{j_{i}}} v(f_{i}(x,\xi)) \right\}, $$
where $I_{a}$, $I_{b}$, $I_{f}$ are finite (possibly empty) sets
of indices, $a_{i}$, $b_{i}$, $f_{i}$ are $\mathcal{L}$-definable
functions, $\nu_{i},M \in \mathbb{N}$ are positive integers,
$\lhd_{1}$, $\lhd_{2}$ stand for $<$ or $\leq$, the predicates
$$ \diamond \in \{ \equiv_{M}, \neg \equiv_{M}, \equiv_{M}^{m'}, \neg \equiv_{M}^{m'} \}
   \ \ \text{with some} \ \ m' \in \mathbb{N}, $$
and $\theta_{j_{i}}$ is one of the entries of $\bar{\theta}$.

\vspace{1ex}

As before, since every $\mathcal{L}$-definable subset in the
Cartesian product $\Gamma^{n} \times \Bbbk^{m}$ of auxiliary sorts
is a finite union of the Cartesian products of definable subsets
in $\Gamma^{n}$ and in $\Bbbk^{m}$, we can assume that $B$ is one
fiber $F(\xi')$ for a parameter $\xi' \in \Bbbk^{m}$. For
simplicity, we abbreviate
$$ c(x,\xi'), a(x,\xi'), b(x,\xi'), a_{i}(x,\xi'), b_{i}(x,\xi'), f_{i}(x,\xi') $$
to
$$ c(x), a(x), b(x), a_{i}(x), b_{i}(x), f_{i}(x) $$
with $i \in I_{a}$, $i \in I_{b}$ and $i \in I_{f}$. Denote by $E
\subset K$ the common domain of these functions; then $0$ is an
accumulation point of $E$.

\vspace{1ex}

By the theorem on existence of the limit (Theorem~\ref{limit-th}),
we can assume that the limits
$$ c(0), a(0), b(0), a_{i}(0), b_{i}(0), f_{i}(0) $$
of the functions
$$ c(x), a(x), b(x), a_{i}(x), b_{i}(x), f_{i}(x) $$
when $x \rightarrow 0\,$ exist in $R$. Moreover, there is a
neighborhood $U$ of $0$ such that, each definable set
$$ \{ (v(x), v(f_{i}(x))): \; x \in (E \cap U) \setminus \{0 \} \}
   \subset \Gamma \times (\Gamma \cup \ \{
   \infty \}), \ \ i \in I_{f},  $$
is contained in an affine line with rational slope
\begin{equation}\label{affine}
q \cdot l = p_{i} \cdot k + \beta_{i}, \ \ i \in I_{f},
\end{equation}
with $p_{i},q \in \mathbb{Z}$, $q>0$, $\beta_{i} \in \Gamma$, or
in\/ $\Gamma \times \{ \infty \}$.

\vspace{1ex}

The role of the center $c(x)$ is, of course, immaterial. We may
assume, without loss of generality, that it vanishes, $c(x) \equiv
0$, for if a point $b = (0,w) \in K^{2}$ lies in the closure of
the cell with zero center, the point $(0, w + c(0))$ lies in the
closure of the cell with center $c(x)$.

\vspace{1ex}

Observe now that If $\lhd_{1}$ occurs and $a(0)= 0$, the set
$F(\xi')$ is itself an $x$-fiber shrinking at $(0,0)$ and the
point $b=(0,0)$ is an accumulation point of $B$ lying over $a=0$,
as desired. And so is the point $b=(0,0)$ if
$\lhd_{1,\theta_{j_{i}}}$ occurs and $a_{i}(0)= 0$ for some $i \in
I_{a}$, because then the set $F(\xi')$ contains the $x$-fiber
shrinking
$$ F(\xi') \cap \{ (x,y) \in E \times K: \
   v(a_{i}(x)) \lhd_{1} v(y^{\nu_{i}}) \}. $$

\vspace{1ex}

So suppose that either only $\lhd_{2}$ occur or $\lhd_{1}$ occur
and, moreover, $a(0) \neq 0$ and $a_{i}(0) \neq 0$ for all $i \in
I_{a}$. By elimination of $K$-quantifiers, the set $v(E)$ is a
definable subset of $\Gamma$. Further, it is easy to check,
applying Theorem~\ref{RQE} ff.\ likewise as it was in
Lemma~\ref{line}, that the set $v(E)$ is given near infinity only
by finitely many parametrized congruence conditions of the form
\begin{equation}\label{vE}
  v(E) = \left\{ k \in \Gamma: \ k > \beta \ \wedge \ \exists \
  \bar{\theta} \ \, \omega(\bar{\theta}) \ \wedge \
  \bigwedge_{i=1}^{s} \ m_{i} k \, \diamond_{N,\theta_{j_{i}}} \gamma_{i}
  \right\}.
\end{equation}
where $\beta, \gamma_{i} \in \Gamma$, $m_{i},N \in \mathbb{N}$ for
$i=1,\ldots,s$, the predicates
$$ \diamond \in \{ \equiv_{N}, \neg \equiv_{N}, \equiv_{N}^{m'}, \neg \equiv_{N}^{m'}
   \} \ \ \text{with some} \ \ m' \in \mathbb{N}, $$
and $\theta_{j_{i}}$ is one of the entries of $\bar{\theta}$.
Obviously, after perhaps shrinking the neighborhood of zero, we
may assume that
$$ v(a(x)) = v(a(0)) \ \ \text{and} \ \ v(a_{i}(x)) = v(a_{i}(0)) $$
for all $i \in I_{a}$ and $x \in E \setminus \{ 0 \}$, $v(x)
> \beta$.

\vspace{1ex}

Now, take an element $(u,w) \in F(\xi')$ with $u \in E \setminus
\{ 0 \}$, $v(u) > \beta$. In order to complete the proof, it
suffices to show that $(0,w)$ is an accumulation point of
$F(\xi')$. To this end, observe that, by equality~\ref{vE}, there
is a point $x \in E$ arbitrarily close to $0$ such that
$$ v(x) \in v(u) + q M N \cdot \Gamma. $$
By equality~\ref{affine}, we get
$$ v(f_{i}(x)) \in v(f_{i}(u)) + p_{i} M N \cdot \Gamma, \ \ \ i \in I_{f}, $$
and hence
\begin{equation}\label{vf}
   v\left( f_{i}(x) \right) \equiv_{M} v(f_{i}(u)), \ \ \ i \in I_{f}.
\end{equation}
Clearly, in the vicinity of zero we have
$$ v(y^{\nu}) \lhd_{2} v(b(x,\xi)) $$
and
$$ \bigwedge_{i \in I_{b}} \
   v(y^{\nu_{i}}) \lhd_{2,\theta_{j_{i}}} v(b_{i}(x,\xi)). $$
Therefore equality~\ref{vf} along with the definition of the fibre
$F(\xi')$ yield $(x,w) \in F(\xi')$, concluding the proof of the
closedness theorem.

\section{Piecewise continuity of definable functions}

Further, let $\mathcal{L}$ be the three-sorted language
$\mathcal{L}$ of Denef--Pas. The main purpose of this section is
to prove the following

\begin{theorem}\label{piece}
Let $A \subset K^{n}$ and $f: A \to \mathbb{P}^{1}(K)$ be an
$\mathcal{L}$-definable function. Then $f$ is piecewise
continuous, i.e.\ there is a finite partition of $A$ into
$\mathcal{L}$-definable locally closed subsets
$A_{1},\ldots,A_{s}$ of $K^{n}$ such that the restriction of $f$
to each $A_{i}$ is continuous.
\end{theorem}

We immediately obtain

\begin{corollary}
The conclusion of the above theorem holds for any
$\mathcal{L}$-definable function $f: A \to K$.
\end{corollary}

The proof of Theorem~\ref{piece} relies on two basic ingredients.
The first one is concerned with a theory of algebraic dimension
and decomposition of definable sets into a finite union of locally
closed definable subsets we begin with. It was established by van
den Dries~\cite{Dries} for certain expansions of rings (and
Henselian valued fields, in particular) which admit quantifier
elimination and are equipped with a topological system. The second
one is the closedness theorem (Theorem~\ref{clo-th}).

\vspace{1ex}

Consider an infinite integral domain $D$ with quotient field $K$.
One of the fundamental concepts introduced by van den
Dries~\cite{Dries} is that of a \emph{topological system} on a
given expansion $\mathcal{D}$ of a domain $D$ in a language
$\widetilde{\mathcal{L}}$. That concept incorporates both
Zariski-type and definable topologies. We remind the reader that
it consists of a topology $\tau_{n}$ on each set $D^{n}$, $n \in
\mathbb{N}$, such that:

1) For any $n$-ary $\widetilde{\mathcal{L}}_{D}$-terms
$t_{1},\ldots,t_{s}$, $n,s \in \mathbb{N}$, the induced map
$$ D^{n} \ni a \longrightarrow (t_{1}(a),\ldots,t_{s}(a)) \in D^{s}
$$
is continuous.

2) Every singleton $\{ a \}$, $a \in D$, is a closed subset of
$D$.

3) For any $n$-ary relation symbol $R$ of the language
$\widetilde{\mathcal{L}}$ and any sequence $1 \leq i_{1} < \ldots
< i_{k} \leq n$, $1 \leq k \leq n$, the two sets
$$ \{ (a_{i_{1}},\ldots,a_{i_{k}}) \in D^{k}: \ \mathcal{D} \models
   R((a_{i_{1}},\ldots,a_{i_{k}})^{\&}), \, a_{i_{1}} \neq 0, \ldots,
   a_{i_{k}} \neq 0 \}, $$
$$ \{ (a_{i_{1}},\ldots,a_{i_{k}}) \in D^{k}: \ \mathcal{D} \models
   \neg R((a_{i_{1}},\ldots,a_{i_{k}})^{\&}), \, a_{i_{1}} \neq 0, \ldots,
   a_{i_{k}} \neq 0 \} $$
are open in $D^{k}$; here $(a_{i_{1}},\ldots,a_{i_{k}})^{\&}$
denotes the element of $D^{n}$ whose $i_{j}$-th coordinate is
$a_{i_{j}}$, $j=1,\ldots,k$, and whose remaining coordinates are
zero.

\vspace{1ex}

Finite intersections of closed sets of the form
$$ \{ a \in D^{n}: t(a)=0 \}, $$
where $t$ is an $n$-ary $\widetilde{\mathcal{L}}_{D}$-term, will
be called \emph{special closed subsets} of $D^{n}$. Finite
intersections of open sets of the form
$$ \{ a \in D^{n}: t(a) \neq 0 \}, $$
$$ \{ a \in D^{n}: \ \mathcal{D} \models
   R((t_{i_{1}}(a),\ldots,t_{i_{k}}(a))^{\&}), \, t_{i_{1}}(a) \neq 0, \ldots,
   t_{i_{k}}(a) \neq 0 \} $$
or
$$ \{ a \in D^{n}: \ \mathcal{D} \models
   \neg R((t_{i_{1}}(a),\ldots,t_{i_{k}}(a))^{\&}), \, t_{i_{1}}(a) \neq 0, \ldots,
   t_{i_{k}}(a) \neq 0 \}, $$
where $t,t_{i_{1}},t_{i_{k}}$ are
$\widetilde{\mathcal{L}}_{D}$-terms, will be called \emph{special
open subsets} of $D^{n}$. Finally, an intersection of a special
open and a special closed subsets of $D^{n}$ will be called a
\emph{special locally closed} subset of $D^{n}$. Every
quantifier-free $\widetilde{\mathcal{L}}$-definable set is a
finite union of special locally closed sets.

\vspace{1ex}

Suppose now that the language $\widetilde{\mathcal{L}}$ extends
the language of rings and has no extra function symbols of arity
$>0$ and that an $\widetilde{\mathcal{L}}$-expansion $\mathcal{D}$
of the domain $D$ under study admits quantifier elimination and is
equipped with a topological system such that every non-empty
special open subset of $D$ is infinite. These conditions ensure
that $\mathcal{D}$ is algebraically bounded and algebraic
dimension is a dimension function on $\mathcal{D}$
(\cite[Proposition~2.15 and~2.7]{Dries}). Algebraic dimension is
the only dimension function on $\mathcal{D}$ whenever, in
addition, $D$ is a non-trivially valued field and the topology
$\tau_{1}$ is induced by its valuation. Then, for simplicity, the
algebraic dimension of an $\widetilde{\mathcal{L}}$-definable set
$E$ will be denoted by $\dim E$.

\vspace{1ex}

Now we recall the following two basic results from the
paper~\cite[Propositions~2.17 and~2.23]{Dries}:

\begin{proposition}\label{partition}
Every $\widetilde{\mathcal{L}}$-definable subset of $D^{n}$ is a
finite union of intersections of Zariski closed with special open
subsets of $D^{n}$ and, a fortiori, a finite union of locally
closed $\widetilde{\mathcal{L}}$-definable subsets of $D^{n}$.
\end{proposition}

\begin{proposition}\label{front-eq}
Let $E$ be an $\widetilde{\mathcal{L}}$-definable subset of
$D^{n}$, and let $\overline{E}$ stand for its closure and
$\partial E := \overline{E} \setminus E$ for its frontier. Then
$$ \mathrm{alg.dim}\, (\partial E) < \mathrm{alg.dim}\, (E). $$
\end{proposition}

It is not difficult to strengthen the former proposition as
follows.

\begin{corollary}\label{part1}
Every $\widetilde{\mathcal{L}}$-definable set is a finite disjoint
union of locally closed sets.
\end{corollary}

Quantifier elimination due to Pas~\cite[Theorem~4.1]{Pa1} (more
precisely, elimination of $K$-quantifiers) enables translation of
the language $\mathcal{L}$ of Denef--Pas on $K$ into a language
$\widetilde{\mathcal{L}}$ described above, which is equipped with
the topological system wherein $\tau_{n}$ is the $K$-topology on
$K^{n}$, $n \in \mathbb{N}$. Indeed, we must augment the language
of rings by adding extra relation symbols for the inverse images
under the valuation and angular component map of relations on the
value group and residue field, respectively. More precisely, we
must add the names of sets of the form
$$ \{ a \in K^{n}: (v(a_{1}),\ldots,v(a_{n})) \in P \} $$
and
$$ \{ a \in K^{n}: (\overline{ac}\, a_{1},\ldots,\overline{ac}\,
   a_{n}) \in Q \}, $$
where $P$ and $Q$ are definable subsets of $\Gamma^{n}$ and
$\Bbbk^{n}$ (as the auxiliary sorts of the language
$\mathcal{L}$), respectively.

\vspace{1ex}

Summing up, the foregoing results apply in the case of Henselian
non-trivially valued fields with the three-sorted language
$\mathcal{L}$ of Denef--Pas. Now we can readily prove
Theorem~\ref{piece}.

\begin{proof}
Consider an $\mathcal{L}$-definable function $f: A \to
\mathbb{P}^{1}(K)$ and its graph
$$ E := \{ (x,f(x)): x \in A \} \subset K^{n} \times \mathbb{P}^{1}(K). $$
We shall proceed with induction with respect to the dimension
$$ d = \dim A = \dim \, E $$
of the source and graph of $f$. By Corollary~\ref{part1}, we can
assume that the graph $E$ is a locally closed subset of $K^{n}
\times \mathbb{P}^{1}(K)$ of dimension $d$ and that the conclusion
of the theorem holds for functions with source and graph of
dimension $< d$.

\vspace{1ex}

Let $F$ be the closure of $E$ in $K^{n} \times \mathbb{P}^{1}(K)$
and $\partial E := F \setminus E$ be the frontier of $E$. Since
$E$ is locally closed, the frontier $\partial E$ is a closed
subset of $K^{n} \times \mathbb{P}^{1}(K)$ as well. Let
$$ \pi: K^{n} \times \mathbb{P}^{1}(K) \longrightarrow K^{n} $$
be the canonical projection. Then, by virtue of the closedness
theorem, the images $\pi(F)$ and $\pi(\partial E)$ are closed
subsets of $K^{n}$. Further,
$$ \dim \, F = \dim \, \pi(F) = d $$
and
$$ \dim \, \pi(\partial E) \leq \dim \, \partial E < d; $$
the last inequality holds by Proposition~\ref{front-eq}. Putting
$$ B := \pi(F) \setminus \pi(\partial E) \subset \pi(E) = A, $$
we thus get
$$ \dim \, B = d \ \ \text{and} \ \ \dim \, (A \setminus B) < d.
$$
Clearly, the set
$$ E_{0} := E \cap (B \times \mathbb{P}^{1}(K)) = F \cap (B \times
   \mathbb{P}^{1}(K)) $$
is a closed subset of $B \times \mathbb{P}^{1}(K)$ and is the
graph of the restriction
$$ f_{0}: B \longrightarrow \mathbb{P}^{1}(K) $$
of $f$ to $B$. Again, it follows immediately from the closedness
theorem that the restriction
$$ \pi_{0} : E_{0} \longrightarrow B $$
of the projection $\pi$ to $E_{0}$ is a definably closed map.
Therefore $f_{0}$ is a continuous function. But, by the induction
hypothesis, the restriction of $f$ to $A \setminus B$ satisfies
the conclusion of the theorem, whence so does the function $f$.
This completes the proof.
\end{proof}

\section{Curve selection}

We now pass to curve selection over non-locally compact ground
fields under study. While the real version of curve selection goes
back to the papers~\cite{B-Car,Wal} (see
also~\cite{Loj,Miln,BCR}), the $p$-adic one was achieved in the
papers~\cite{S-Dries,De-Dries}.

In this section we give two versions of curve selection which are
counterparts of the ones from our paper~\cite[Proposition~8.1
and~8.2]{Now2} over rank one valued fields. The first one is
concerned with valuative semialgebraic sets and we can repeat
verbatim its proof which relies on transformation to a normal
crossing by blowing up and the closedness theorem.


\vspace{1ex}

By a valuative semialgebraic subset of $K^{n}$ we mean a (finite)
Boolean combination of elementary valuative semialgebraic subsets,
i.e.\ sets of the form
$$ \{ x \in K^{n}: \ v(f(x)) \leq v(g(x)) \}, $$
where $f$ and $g$ are regular functions on $K^{n}$. We call a map
$\varphi$ semialgebraic if its graph is a valuative semialgebraic
set.


\begin{proposition}\label{CSL}
Let $A$ be a valuative semialgebraic subset of $K^{n}$. If a point
$a \in K^{n}$ lies in the closure (in the $K$-topology) of $A
\setminus \{ a \}$, then there is a semialgebraic map $\varphi : R
\longrightarrow K^{n}$ given by algebraic power series such that
$$ \varphi(0)=a \ \ \ {\text and} \ \ \  \varphi( R
   \setminus \{ 0 \}) \subset A \setminus \{ a \}. $$
\end{proposition}

\vspace{1ex}

We now turn to the general version of curve selection for
$\mathcal{L}$-definable sets. Under the circumstances, we apply
relative quantifier elimination in a many-sorted language due to
Cluckers--Halupczok rather than simply quantifier elimination in
the Presburger language for rank one valued fields. The passage
between the two corresponding reasonings for curve selection is
similar to that for fiber shrinking. Nevertheless we provide a
detailed proof for more clarity and the reader's convenience. Note
that both fiber shrinking and curve selection apply
Lemma~\ref{line}.

\begin{proposition}\label{GCSL}
Let $A$ be an $\mathcal{L}$-definable subset of $K^{n}$. If a
point $a \in K^{n}$ lies in the closure (in the $K$-topology) of
$A \setminus \{ a \}$, then there exist a semialgebraic map
$\varphi : R \longrightarrow K^{n}$ given by algebraic power
series and an $\mathcal{L}$-definable subset $E$ of $R$ with
accumulation point $0$ such that
$$ \varphi(0)=a \ \ \ {\text and} \ \ \  \varphi( E
   \setminus \{ 0 \}) \subset A \setminus \{ a \}. $$
\end{proposition}

\begin{proof}
As before, we proceed with induction with respect to the dimension
of the ambient space $n$. The case $n=1$ being evident, suppose
$n>1$. By elimination of $K$-quantifiers, the set $A \setminus \{
a \}$ is a finite union of sets defined by conditions of the form
$$ (v(f_{1}(x)),\ldots,v(f_{r}(x))) \in P, \ \
   (\overline{ac}\, g_{1}(x),\ldots,\overline{ac}\, g_{s}(x)) \in Q,
$$
where $f_{i},g_{j} \in K[x]$ are polynomials, and $P$ and $Q$ are
definable subsets of $\Gamma^{r}$ and $\Bbbk^{s}$, respectively.
Without loss of generality, we may assume that $A$ is such a set
and $a=0$.

\vspace{1ex}

Take a finite composite
$$ \sigma: Y \longrightarrow K\mathbb{A}^{n} $$
of blow-ups along smooth centers such that the pull-backs
$$ f_{1}^{\sigma},\ldots, f_{r}^{\sigma} \ \ \ \text{and}
   \ \ \  g_{1}^{\sigma},\ldots, g_{s}^{\sigma} $$
are normal crossing divisors unless they vanish. Since the
restriction $\sigma: Y(K) \longrightarrow K^{n}$ is definably
closed (Corollary~\ref{clo-th-cor-3}), there is a point $b \in
Y(K) \cap \sigma^{-1}(a)$ which lies in the closure of the set
$$ B := Y(K) \cap \sigma^{-1}(A \setminus \{ a \}). $$
Take local coordinates $y_{1}.\ldots,y_{n}$ near $b$ in which
$b=0$ and every pull-back above is a normal crossing. We shall
first select a semialgebraic map $\psi : R \longrightarrow Y(K)$
given by restricted power series and an $\mathcal{L}$-definable
subset $E$ of $R$ with accumulation point $0$ such that
$$ \psi(0)=b \ \ \ {\text and} \ \ \  \psi( E
   \setminus \{ 0 \}) \subset B. $$

\vspace{1ex}

Since the valuation map and the angular component map composed
with a continuous function are locally constant near any point at
which this function does not vanish, the conditions which describe
the set $B$ near $b$ are of the form
$$ (v(y_{1}),\ldots,v(y_{n})) \in \widetilde{P}, \ \
   (\overline{ac}\, y_{1},\ldots,\overline{ac}\, y_{n}) \in
  \widetilde{Q}, $$
where $\widetilde{P}$ and $\widetilde{Q}$ are definable subsets of
$\Gamma^{n}$ and $\Bbbk^{n}$, respectively.

\vspace{1ex}

The set $B_{0}$ determined by the conditions
$$ (v(y_{1}),\ldots,v(y_{n})) \in \widetilde{P}, $$
$$ (\overline{ac}\, y_{1},\ldots,\overline{ac}\, y_{n}) \in \widetilde{Q} \, \cap
   \, \bigcup_{i=1}^{n} \, \{ \xi_{i}=0 \}, $$
is contained near $b$ in the union of hyperplanes $\{ y_{i}=0 \}$,
$i=1,\ldots,n$. If $b$ is an accumulation point of the set
$B_{0}$, then the desired map $\psi$ exists by the induction
hypothesis. Otherwise $b$ is an accumulation point of the set
$B_{1} := B \setminus B_{0}$.

\vspace{1ex}

Now we are going to apply relative quantifier elimination in the
value group sort $\Gamma$. Similarly, as in the proof of
Lemma~\ref{line}, the parametrized congruence conditions which
occur in the description of the definable subset $\widetilde{P}$
of $\Gamma^{n}$, achieved via quantifier elimination, are not an
essential obstacle to finding the desired map $\psi$, but affect
only the definable subset $E$ of $R$. Neither are the conditions
$$ \widetilde{Q} \, \setminus \, \bigcup_{i=1}^{n} \, \{ \xi_{i}=0 \}$$
imposed on the angular components of the coordinates
$y_{1},\ldots,y_{n}$, because none of them vanishes here.
Therefore, in order to select the map $\psi$, we must first of all
analyze the linear conditions (equalities and inequalities) which
occur in the description of the set $\widetilde{P}$.

\vspace{1ex}

The set $\widetilde{P}$ has an accumulation point
$(\infty,\ldots,\infty)$ as $b=0$ is an accumulation point of $B$.
By Lemma~\ref{line}, there is an affine semi-line
$$ L = \{ (r_{1}t + \gamma_{1},\ldots,r_{n}t + \gamma_{n}): \, t
   \in \Gamma, \ t \geq 0 \} \ \ \ \text{with} \ \ r_{1},\ldots,r_{n} \in \mathbb{N}, $$
passing through a point $\gamma = (\gamma_{1},\ldots,\gamma_{n})
\in P$ and such that $(\infty,\ldots,\infty)$ is an accumulation
point of the intersection $P \cap L$ too.

\vspace{1ex}

Now, take some elements
$$ (\xi_{1},\ldots,\xi_{n}) \in \widetilde{Q} \, \setminus
   \, \bigcup_{i=1}^{n} \, \{ \xi_{i}=0 \} $$
and next some elements $w_{1},\dots,w_{n} \in K$ for which
$$ v(w_{1})=\gamma_{1},\ldots,v(w_{n})=\gamma_{n} \ \ \ \text{and}
   \ \ \ \overline{ac}\, w_{1} = \xi_{1}, \ldots, \overline{ac}\, w_{n}
   = \xi_{n}. $$
It is not difficult to check that there exists an
$\mathcal{L}$-definable subset $E$ of $R$ which is determined by a
finite number of parametrized congruence conditions (in the
many-sorted language $\mathcal{L}_{qe}$ described in Section~7)
imposed on $v(t)$ and the conditions $\overline{ac}\, t =1$ such
that the subset
$$ F := \left\{ \left( w_{1} \cdot t^{r_{1}}, \ldots, w_{n} \cdot t^{r_{n}} \right):
   \; t \in E \right\} $$
of the arc
$$ \psi: R \to Y, \ \ \psi(t) = \left( w_{1} \cdot t^{r_{1}}, \ldots, w_{n} \cdot t^{r_{n}}
   \right) $$
is contained in $B_{1}$. Then $\varphi := \sigma \circ \psi$ is
the map we are looking for. This completes the proof.
\end{proof}

\section{The \L{}ojasiewicz inequalities}

In this section we provide certain two versions of the
\L{}ojasiewicz inequality which generalize the ones from
\cite[Propositions~9.1 and~9.2]{Now2} to the case of arbitrary
Henselian valued fields. Moreover, the first one is now formulated
for several functions $g_{1},\ldots,g_{m}$. For its proof we still
need the following easy consequence of the closedness theorem.

\begin{proposition}\label{bound}
Let $f: A \to K$ be a continuous $\mathcal{L}$-definable function
on a closed bounded subset $A \subset K^{n}$. Then $f$ is a
bounded function, i.e.\ there is an $\omega \in \Gamma$ such that
$v(f(x)) \geq \omega$ for all $x \in A$.
\end{proposition}

We adopt the following notation:
$$ v(x) = v(x_{1},\ldots,x_{n}) := \min \: \{ v(x_{1}), \dots,
   v(x_{n}) \}
$$
for $x = (x_{1},\ldots,x_{n}) \in K^{n}$.

\begin{theorem}\label{Loj1}
Let $f,g_{1},\ldots,g_{m}: A \to K$ be continuous
$\mathcal{L}$-definable functions on a closed (in the
$K$-topology) bounded subset $A$ of $K^{m}$. If
$$ \{ x \in A: g_{1}(x)= \ldots =g_{m}(x) =0 \} \subset \{ x \in A: f(x)=0 \}, $$
then there exist a positive integer $s$ and a constant $\beta \in
\Gamma$ such that
$$ s \cdot v(f(x)) + \beta \geq v((g_{1}(x), \ldots ,g_{m}(x))) $$
for all $x \in A$.
\end{theorem}

\begin{proof}
Put $g = (g_{1},\ldots,g_{m})$. It is easy to check that the set
$$ A_{\gamma} := \{ x \in A: \; v(f(x))=\gamma \} $$
is a closed $\mathcal{L}$-definable subset of $A$ for every
$\gamma \in \Gamma$. By the hypothesis and the closedness theorem,
the set $g(A_{\gamma})$ is a closed $\mathcal{L}$-definable subset
of $K^{m} \setminus \{ 0 \}$, $\gamma \in \Gamma$. The set
$v(g(A_{\gamma}))$ is thus bounded from above, i.e.\
$$ v(g(A_{\gamma})) \leq \alpha (\gamma) $$
for some $\alpha(\gamma) \in \Gamma$. By elimination of
$K$-quantifiers, the set
$$ \Lambda := \{ (v(f(x)),v(g(x))) \in \Gamma^{2}: \; x \in A, \ f(x) \neq 0
   \} $$
$$ \subset \{ (\gamma,\delta) \in \Gamma^{2}: \; \delta \leq
\alpha(\gamma) \} $$ is a definable subset of $\Gamma^{2}$ in the
many-sorted language $\mathcal{L}_{qe}$ from Section~7. Applying
Theorem~\ref{RQE} ff., we see that this set is described by a
finite number of parametrized linear equalities and inequalities,
and of parametrized congruence conditions. Hence
$$ \Lambda \, \cap  \{ (\gamma,\delta) \in \Gamma^{2}: \; \gamma > \gamma_{0} \}
   \subset \{ (\gamma,\delta) \in \Gamma^{2}: \; \delta \leq s \cdot \gamma \} $$
for a positive integer $s$ and some $\gamma_{0} \in \Gamma$. We
thus get
$$ v(g(x)) \leq s \cdot v(f(x)) \ \ \text{if} \ \ x \in A, \, v(f(x))> \gamma_{0}. $$
Again, by the hypothesis, we have
$$ g(\{ x \in A:  \ v(f(x)) \leq \gamma_{0} \}) \subset K^{m} \setminus \{ 0
   \}. $$
Therefore it follows from the closedness theorem that the set
$$ \{ v(g(x)) \in \Gamma: \ v(f(x)) \leq \gamma_{0} \} $$
is bounded from above, say, by a $\theta \in \Gamma$. Taking an
$\omega \in \Gamma$ as in Proposition~\ref{bound} and putting
$\beta := \max \, \{ 0, \, \theta -s \cdot \omega \}$, we get
$$ s \cdot v(f(x)) - v(g(x)) + \beta \geq 0, \ \ \text{for all} \ \ x \in A, $$
as desired.
\end{proof}

A direct consequence of Theorem~\ref{Loj1} is the following result
on H\"{o}lder continuity of definable functions.

\begin{proposition}\label{Hol}
Let $f: A \to K$ be a continuous $\mathcal{L}$-definable function
on a closed bounded subset $A \subset K^{n}$. Then $f$ is
H\"{o}lder continuous with a positive integer $s$ and a constant
$\beta \in \Gamma$, i.e.\
$$ s \cdot v(f(x) - f(z)) + \beta \geq  v(x-z) $$
for all $x,z \in A$.
\end{proposition}

\begin{proof}
Apply Theorem~\ref{Loj1} to the functions
$$ f(x) - f(y) \ \ \text{and} \ \ g_{i}(x,y)=x_{i} - y_{i}, \
   i=1,\ldots,n. $$
\end{proof}

We immediately obtain

\begin{corollary}
Every continuous $\mathcal{L}$-definable function $f: A \to K$ on
a closed bounded subset $A \subset K^{n}$ is uniformly continuous.
\end{corollary}

Now we state a version of the \L{}ojasiewicz inequality for
continuous definable functions of a locally closed subset of
$K^{n}$.

\begin{theorem}\label{Loj2}
Let $f,g: A \to K$ be two continuous $\mathcal{L}$-definable
functions on a locally closed subset $A$ of $K^{n}$. If
$$ \{ x \in A: g(x)=0 \} \subset \{ x \in A: f(x)=0 \}, $$
then there exist a positive integer $s$ and a continuous
$\mathcal{L}$-definable function $h$ on $A$ such that $f^{s}(x) =
h(x) \cdot g(x)$ for all $x \in A$.
\end{theorem}

\begin{proof}
It is easy to check that the set $A$ is of the form $A := U \cap
F$, where $U$ and $F$ are two $\mathcal{L}$-definable subsets of
$K^{n}$, $U$ is open and $F$ is closed in the $K$-topology.

\vspace{1ex}

We shall adapt the foregoing arguments. Since the set $U$ is open,
its complement $V:= K^{n} \setminus U$ is closed in $K^{n}$ and
$A$ is the following union of open and closed subsets of $K^{n}$
and of $\mathbb{P}^{n}(K)$:
$$ X_{\beta} := \{ x \in K^{n}: \; v(x_{1}),\ldots,v(x_{n}) \geq
   -\beta, $$
$$  v(x-y) \leq \beta \ \ \ \text{for all} \ \ y \in V \}, $$
where $\beta \in \Gamma$, $\beta \geq 0$. As before, we see that
the sets
$$ A_{\beta,\gamma} := \{ x \in X_{\beta}: \; v(f(x))=\gamma \} \ \
   \text{with} \ \beta, \gamma \in \Gamma $$
are closed $\mathcal{L}$-definable subsets of $\mathbb{P}^{n}(K)$,
and next that the sets $g(A_{\beta,\gamma})$ are closed
$\mathcal{L}$-definable subsets of $K \setminus \{ 0 \}$ for all
$\beta,\gamma \in \Gamma$. Likewise, we get
$$ \Lambda := \{ (\beta,v(f(x)),v(g(x))) \in \Gamma^{3}: \; x \in X_{\beta}, \ f(x) \neq 0 \}
   \subset $$
$$ \subset \{ (\beta,\gamma,\delta) \in \Gamma^{3}: \; \delta <
   \alpha(\beta,\gamma) \} $$
for some $\alpha(\beta,\gamma) \in \Gamma$.

\vspace{1ex}

$\Lambda$ is a definable subset of $\Gamma^{3}$ in the many-sorted
language $\mathcal{L}_{qe}$, and thus is described by a finite
number of parametrized linear equalities and inequalities, and of
parametrized congruence conditions. Again, the above inclusion
reduces to an analysis of those linear equalities and
inequalities. Consequently, there exist a positive integer $s \in
\mathbb{N}$ and elements $\gamma_{0}(\beta) \in \Gamma$ such that
$$ \Lambda \cap \{ (\beta,\gamma,\delta) \in \Gamma^{3}: \;
   \gamma > \gamma_{0}(\beta) \} \subset \{ (\beta,\gamma,\delta) \in \Gamma^{3}:
   \; \delta < s \cdot \gamma \}. $$
Since $A$ is the union of the sets $X_{\beta}$, it is not
difficult to check that the quotient $f^{s}/g$ extends by zero
through the zero set of the denominator to a (unique) continuous
$\mathcal{L}$-definable function on $A$, which is the desired
result.
\end{proof}

We conclude this section with a theorem which is much stronger
than its counterpart, \cite[Proposition~12.1]{Now2}, concerning
continuous rational functions. The proof we give now resembles the
above one, without applying transformation to a normal crossing.
Put
$$ \mathcal{D}(f) := \{ x \in A: f(x)
   \neq 0 \} \ \ \text{and} \ \ \mathcal{Z}\, (f) := \{ x \in A: f(x) = 0 \}. $$

\begin{theorem}\label{Loj3}
Let $f: A \to K$ be a continuous $\mathcal{L}$-definable function
on a locally closed subset $A$ of $K^{n}$ and $g: \mathcal{D}(f)
\to K$ a continuous $\mathcal{L}$-definable function. Then $f^{s}
\cdot g$ extends, for $s \gg 0$, by zero through the set
$\mathcal{Z}\, (f)$ to a (unique) continuous
$\mathcal{L}$-definable function on $A$.
\end{theorem}

\begin{proof}
As in the proof of Theorem~\ref{Loj2}, let $A = U \cap F$ and
consider the same sets $X_{\beta} \subset K^{n}$, $\beta \in
\Gamma$, and $\Lambda \subset \Gamma^{3}$. Under the assumptions,
we get
$$ \Lambda \subset \{ (\beta,\gamma,\delta) \in \Gamma^{3}: \; \delta
   > \alpha(\beta,\gamma) \} $$
for some $\alpha(\beta,\gamma) \in \Gamma$. Now, in a similar
fashion as before, we can find an integer $r \in \mathbb{Z}$ and
elements $\gamma_{0}(\beta) \in \Gamma$ such that
$$ \Lambda \cap \{ (\beta,\gamma,\delta) \in \Gamma^{3}: \;
   \gamma > \gamma_{0}(\beta) \} \subset \{ (\beta,\gamma,\delta) \in \Gamma^{3}:
   \; \delta > r \cdot \gamma \}. $$
Take a positive integer $s \in \mathbb{N}$ such that $s + r >0$.
Then, as in the proof of Theorem~\ref{Loj2}, it is not difficult
to check that the function $f^{s} \cdot g$ extends by zero through
the zero set of $f$ to a (unique) continuous
$\mathcal{L}$-definable function on $A$, which is the desired
result.
\end{proof}

\begin{remark}
Note that Theorem~\ref{Loj3} is, in fact, a strengthening of
Theorem~\ref{Loj2}, and has many important applications. In
particular, it plays a crucial role in the proof of the
Nullstellensatz for regulous (i.e.\ continuous and rational)
functions on $K^{n}$.
\end{remark}


\section{Continuous hereditarily rational functions and regulous functions and sheaves}

Continuous rational functions on singular real algebraic
varieties, unlike those on non-singular real algebraic varieties,
often behave quite unusually. This is illustrated by many examples
from the paper~\cite[Section~1]{K-N}, and gives rise to the
concept of hereditarily rational functions. We shall assume that
the ground field $K$ is not algebraically closed. Otherwise, the
notion of a continuous rational function on a normal variety
coincides with that of a regular function and, in general, the
study of continuous rational functions leads to the concept of
{\it seminormality} and {\it seminormalization};
cf.~\cite{A-B,A-N} or \cite[Section~10.2]{Kol-1} for a recent
treatment. Let $K$ be topological field with the density property.
For a $K$-variety $Z$, let $Z(K)$ denote the set of all $K$-points
on $Z$. We say that a continuous function $f: Z(K) \longrightarrow
K$ is {\it hereditarily rational} if for every irreducible
subvariety $Y \subset Z$ there exists a Zariski dense open
subvariety $Y^{0} \subset Y$ such that $f|_{Y^{0}(K)}$ is regular.
Below we recall an extension theorem, which plays a crucial role
in the theory of continuous rational functions. It says roughly
that continuous rational extendability to the non-singular ambient
space is ensured by (and in fact equivalent to) the intrinsic
property to be continuous hereditarily rational. This theorem was
first proven for real and $p$-adic varieties in~\cite{K-N}, and
next over Henselian rank one valued fields
in~\cite[Section~10]{Now2}. The proof of the latter result relied
on the closedness theorem (Theorem~\ref{clo-th}), the descent
property (Corollary~\ref{clo-th-cor-4}) and the \L{}ojasiewicz
inequality (Theorem~\ref{Loj2}), and can now be repeated verbatim
for the case where $K$ is an arbitrary Henselian valued field $K$
of equicharacteristic zero.

\begin{theorem}\label{ext-th}
Let $X$ be a non-singular $K$-variety and $W \subset Z\subset X$
closed subvarieties. Let $f$ be a continuous hereditarily rational
function on $Z(K)$ that is regular at all $K$-points of $Z(K)
\setminus W(K)$. Then $f$ extends to a continuous hereditarily
rational function $F$ on $X(K)$ that is regular at all $K$-points
of $X(K) \setminus W(K)$.
\end{theorem}

The corresponding theorem for hereditarily rational functions of
class $\mathcal{C}^{k}$, $k \in \mathbb{N}$, remains an open
problem as yet. This leads to the concept of $k$-regulous
functions, $k \in \mathbb{N}$, on a subvariety $Z(K)$ of a
non-singular $K$-variety $X(K)$, i.e.\ those functions on $Z(K)$
which are the restrictions to $Z(K)$ of rational functions of
class $\mathcal{C}^{k}$ on $X(K)$.

\vspace{1ex}

In real algebraic geometry, the theory of regulous functions,
varieties and sheaves was developed by
Fichou--Huisman--Mangolte--Monnier~\cite{FHMM}. Regulous geometry
over Henselian rank one valued fields was studied in our
paper~\cite[Sections~11, 12, 13]{Now2}. The basic tools we applied
are the closedness theorem, descent property, the \L{}ojasiewicz
inequalities and transformation to a normal crossing by blowing
up. We should emphasize that all those our results, including the
Nullstellensatz and Cartan's theorems A and B for regulous
quasi-coherent sheaves, remain true over arbitrary Henselian
valued fields (of equicharacteristic zero) with almost the same
proofs.

\vspace{1ex}

We conclude this paper with the following comment.

\begin{remark}
In our recent paper~\cite{Now-6}, we established a definable,
non-Archimedean version of the closedness theorem over Henselian
valued fields (of equicharacteristic zero) with analytic structure
along with several applications. Let us mention, finally, that the
theory of analytic structures goes back to the work of many
mathematicians (see
e.g.~\cite{De-Dries,Dries-1,Lip,Dries-Mac,Dries-Has,L-R-0,L-R,C-Lip-R,C-Lip-0,C-Lip}).
\end{remark}

\vspace{2ex}

\begin{small}
Institute of Mathematics

Faculty of Mathematics and Computer Science

Jagiellonian University

ul.~Profesora \L{}ojasiewicza 6

30-348 Krak\'{o}w, Poland

{\em E-mail address: nowak@im.uj.edu.pl}
\end{small}


\vspace{2ex}

\begin{thebibliography}{99}

\bibitem{A-B}
A.~Andreotti, E.~Bombieri, {\em Sugli omeomorfismi delle variet\`a
algebriche\/}, Ann. Scuola Norm. Sup Pisa \textbf{23} (3) (1969),
431--450.

\bibitem{A-N}
A.~Andreotti, F.~Norguet, \emph{La convexit\'e holomorphe dans
l'espace analytique des cycles d'une vari\'et\'e alg\'ebrique},
Ann. Scuola Norm. Sup. Pisa \textbf{21} (3) (1967), 31--82.

\bibitem{AM}
M.~Artin, B.~Mazur, \emph{On periodic points}, Ann.\ Math.\ {\bf
81} (1965), 82--99.

\bibitem{BCR}
J.~Bochnak, M.~Coste, M.-F.~Roy, \emph{Real
  Algebraic Geometry}, Ergebnisse der Mathematik und ihrer Grenzgebiete,
  vol.~36, Springer-Verlag, Berlin, 1998.

\bibitem{Bour}
N.~Bourbaki, {\em Alg\`{e}bre Commutative}, Hermann, Paris, 1962.

\bibitem{B-Car}
F.~Bruhat, H.~Cartan, {\em Sur la structure des sous-ensembles
analytiques r\'{e}els\/}, C.\ R.\ Acad.\ Sci.\ {\bf 244} (1957),
988--990.

\bibitem{Ch}
G.~Cherlin, {\em Model Theoretic Algebra, Selected Topics\/},
Lect.\ Notes Math.\ {\bf 521}, Springer-Verlag, Berlin, 1976.

\bibitem{C-H}
R.~Cluckers, E.~Halupczok, {\em Quantifier elimination in ordered
abelian groups\/}, Confluentes Math.\ {\bf 3} (2011), 587--615.

\bibitem{C-Lip-R}
R.~Cluckers, L.~Lipshitz, Z.~Robinson, \emph{Analytic cell
decomposition and analytic motivic integration}, Ann.\ Sci.\ École
Norm.\ Sup.\ (4) {\bf 39} (2006), 535--568.

\bibitem{C-Lip-0}
R.~Cluckers, L.~Lipshitz, {\em Fields with analytic structure\/},
J.\ Eur.\ Math.\ Soc.\ {\bf 13} (2011), 1147--1223.

\bibitem{C-Lip}
R.~Cluckers, L.~Lipshitz, {\em Strictly convergent analytic
structures\/}, J.\ Eur.\ Math.\ Soc.\ {\bf 19} (2017), 107--149.

\bibitem{De-Dries}
J.~Denef, L.~van den Dries, {\em $p$-adic and real subanalytic
sets\/}, Ann.\ Math.\ {\bf 128} (1988), 79--138.

\bibitem{Dries}
L.~van den Dries, \emph{Dimension of definable sets, algebraic
boundedness and Henselian fields}, Ann.\ Pure Appl.\ Logic {\bf
45} (1989), 189--209.

\bibitem{Dries-1}
L.~van den Dries, {\em Analytic Ax--Kochen--Ershov theorems},
Contemporary Mathematics {\bf 131} (1992), 379--392.

\bibitem{Dries-Has}
L.~van den Dries, D.~Haskell, D.~Macpherson, {\em One dimensional
$p$-adic subanalytic sets}, J.\ London Math.\ Soc.\ {\bf 56}
(1999), 1--20.

\bibitem{Dries-Mac}
L.~van den Dries, A.~Macintyre, D.~Marker, {\em The elementary
theory of restricted analytic fields with exponentiation}, Ann.\
Math.\ {\bf 140} (1994), 183--205.

\bibitem{Eis}
D.~Eisenbud, \emph{Commutative Algebra with a View Towards
Algebraic Geometry}, Graduate Texts in Math.\ {\bf 150},
Springer-Verlag, New York, 1994.

\bibitem{E-Pre}
A.J.~Engler, A.~Prestel, \emph{Valued Fields}, Springer-Verlag,
Berlin, 2005.

\bibitem{FHMM}
G.~Fichou, J.~Huisman, F.~Mangolte, J.-P.~Monnier, {\em Fonctions
r\'{e}gulues\/}, J.\ Reine Angew.\ Math.\ {\bf 718} (2016),
103--151.

\bibitem{Fish}
B.~Fisher, {\em A note on Hensel's lemma in several variables\/},
Proc.\ Amer.\ Math.\ Soc.\ {\bf 125} (11) (1997), 3185--3189.

\bibitem{G-G-MB}
O.~Gabber, P.~Gille, L.~Moret-Bailly, \emph{Fibr\'{e}s principaux
sur les corps valu\'{e}s hens\'{e}liens}, Algebraic Geometry {\bf
1} (2014), 573--612.

\bibitem{G-Pop-Roq}
B.~Green, F.~Pop, P.~Roquette, \emph{On Rumely's local global
principle}, Jber.\ d.\ Dt.\ Math.-Verein.\ {\bf 97} (1995),
43--74.

\bibitem{Gro}
A.~Grothendieck, {\em Éléments de Géométrie Algébrique.\ III.\
Étude cohomologique des faisceaux cohérents\/}, Publ.\ Math.\ IHES
{\bf 11} (1961) and {\bf 17} (1963).

\bibitem{Gur}
Y.~Gurevich, \emph{Elementary properties of ordered abelian
groups}, Algebra i Logika Seminar, {\bf 3} (1964), 5--39 (in
Russian); Amer.\ Math.\ Soc.\ Transl., II Ser.\ {\bf 46} (1965),
165--192 (in English).

\bibitem{G-S}
Y.~Gurevich, P.H.~Schmitt, \emph{The theory of ordered abelian
groups does not have the independence property}, Trans.\ Amer.\
Math.\ Soc.\ {\bf 284} (1984), 171-–182.

\bibitem{Ha}
R.~Hartshorne, {\em Algebraic Geometry\/}, Graduate Texts in
Mathematics 52, Springer-Verlag, 1977.

\bibitem{Kap}
I.~Kaplansky, \emph{Maximal fields with valuations I and II}, Duke
Math.\ J.\ {\bf 9} (1942), 303--321 and {\bf 12} (1945), 243--248.

\bibitem{Kol}
J.~Koll{\'a}r, Lectures on resolution of singularities, Ann,\
Math.\ Studies, Vol.\ 166, Princeton Univ.\ Press, Princeton, New
Jersey, 2007.

\bibitem{Kol-1}
J.~Koll{\'a}r, \emph{Singularities of the minimal model program},
(With a collaboration of S.~Kov\'acs) Cambridge Tracts in
Mathematics, vol.\ 200. Cambridge Univ.\ Press, Cambridge, 2013.


\bibitem{K-N}
J.~Koll{\'a}r, K.~Nowak, {\em Continuous rational functions on
real and $p$-adic varieties\/}, Math.\ Zeitschrift {\bf 279}
(2015), 85--97.


\bibitem{K-1}
W.~Kucharz, {\em Approximation by continuous rational maps into
spheres\/}, J.\ Eur.\ Math.\ Soc.\ {\bf 16} (2014), 1555--1569.

\bibitem{K-2}
W.~Kucharz, {\em Continuous rational maps into spheres},  Math.\
Zeitschrift, to appear; DOI 10.1007/s00209-016-1639-4.

\bibitem{K-3}
W.~Kucharz, \emph{Piecewise-regular maps}, Math.\ Ann., to appear;
DOI 10.1007/s00208-017-1607-2.

\bibitem{K-K}
W.~Kucharz, K.~Kurdyka, \emph{Stratified-algebraic vector
bundles}, J.\ Reine Angew.\ Math., to appear; DOI
10.1515/crelle-2015-0105.

\bibitem{K-Z}
W.~Kucharz, M.~Zieli\'nski, \emph{Regulous vector bundles}, arXiv:
1703.05566 [math.AG].

\bibitem{Kuhl}
F.-V.~Kuhlmann, {\em Maps on ultrametric spaces, Hensel's lemma
and differential equations over valued fields\/}, Comm.\ Algebra
{\bf 39} (2011), 1730--1776.

\bibitem{Lip}
L.~Lipshitz, \emph{Rigid subanalytic sets}, Amer.\ J.\ Math.\ {\bf
115} (1993), 77-108.

\bibitem{L-R-0}
L.~Lipshitz, Z.~Robinson, \emph{Rings of Separated Power Series
and Quasi-Affinoid Geometry}, Ast\'{e}risque {\bf 264} (2000).

\bibitem{L-R}
L.~Lipshitz, Z.~Robinson, {\em Uniform properties of rigid
subanalytic sets\/}, Trans.\ Amer.\ Math.\ Soc.\ {357} (11)
(2005), 4349--4377.

\bibitem{Loj}
S.~\L{}ojasiewicz, {\em Ensembles Semi-analytiques}, I.H.E.S.,
Bures-sur-Yvette, 1965.

\bibitem{Miln}
J.~Milnor, {\em Singular points of complex hypersurfaces\/},
Princeton Univ.\ Press, Princeton, New Jersey, 1968.

\bibitem{Now0}
K.J.~Nowak, {\em On the Abhyankar--Jung theorem for Henselian
$k[x]$-algebras of formal power series}, IMUJ Preprint {\bf
2009/02} (2009); available online
www2.im.uj.edu.pl/badania/preprinty/imuj2009/pr0902.pdf.

\bibitem{Now1}
K.J.~Nowak, {\em Supplement to the paper "Quasianalytic
perturbation of multiparameter hyperbolic polynomials and
symmetric matrices" (Ann.\ Polon.\ Math.\ 101 (2011),
275--291)\/}, Ann. Polon. Math. {\bf 103} (2012), 101-107.

\bibitem{Now2}
K.J.~Nowak, {\em Some results of algebraic geometry over Henselian
rank one valued fields\/},  Sel.\ Math.\ New Ser.\ {\bf 23}
(2017), 455--495.


\bibitem{Now3}
K.J.~Nowak, {\em Piecewise continuity of functions definable over
Henselian rank one valued fields}, arXiv:1702.07849 [math.AG].

\bibitem{Now4}
K.J.~Nowak, {\em On functions given by algebraic power series over
Henselian valued fields}, arXiv:1703.08203 [math.AG].

\bibitem{Now5}
K.J.~Nowak, {\em The closedness theorem over Henselian valued
fields}, arXiv:1704.01093 [math.AG].

\bibitem{Now-6}
K.J.~Nowak, {\em Some results of geometry over Henselian fields
with analytic structure}, arXiv:1808.02481 [math.AG] (2018).

\bibitem{Now-7}
K.J.~Nowak, {\em Definable retractions and a non-Archimedean
Tietze--Urysohn theorem over Henselian valued fields},
arXiv:1808.09782 [math.AG] (2018).

\bibitem{Now-8}
K.J.~Nowak, {\em Definable retractions over complete fields with separated power series}, 
arXiv:1901.00162 [math.AG] (2019).

\bibitem{Now-9}
K.J.~Nowak, {\em Definable transformation to normal crossings over
Henselian fields with separated analytic structure}, Symmetry {\bf 11} (7) (2019), 934.


\bibitem{P-R}
A.~Parusi\'{n}ski, G.~Rond, \emph{The Abhyankar--Jung theorem},
J.\ Algebra {\bf 365} (2012), 29-41.


\bibitem{Pa1}
J.~Pas, {\em Uniform p-adic cell decomposition and local zeta
functions\/}, J.\ Reine Angew.\ Math.\ {\bf 399} (1989), 137--172.

\bibitem{Pa2}
J.~Pas, {\em On the angular component map modulo $p$\/}, J.\
Symbolic Logic {\bf 55} (1990), 1125--1129.

\bibitem{P-Z}
A.~Prestel, M.~Ziegler, \emph{Model theoretic methods in the
theory of topological fields}, J.\ Reine Angew.\ Math.\ {\bf
299--300} (1978), 318--341.

\bibitem{Sch}
P.H.~Schmitt, \emph{Model and substructure complete theories of
ordered abelian groups}; In: \emph{Models and Sets} (Proceedings
of Logic Colloquium '83), Lect.\ Notes Math. {\bf 1103},
Springer-Verlag, Berlin, 1984, 389--418.

\bibitem{S-Dries}
P.~Scowcroft, L.~van den Dries, {\em On the structure of
semi-algebraic sets over $p$-adic fields\/}, J.\ Symbolic Logic
{\bf 53} (4) (1988), 1138--1164.

\bibitem{Wal}
A.H.~Wallace, {\em Algebraic approximation of curves\/}, Canadian
J.\ Math.\ {\bf 10} (1958), 242--278.

\bibitem{Z-S}
O.~Zariski, P.~Samuel, {\em Commutative Algebra}, Vol.\ II, Van
Nostrand, Princeton, 1960.

\end{thebibliography}
\end{document}